\documentclass[11pt]{article}
\usepackage{amssymb,amscd,tikz,amsmath,mathrsfs,faktor,tikz-cd}

\usepackage{srcltx}
\usepackage{graphicx}
\usepackage{DeltaComplexesProject}

\newcommand{\QQ}{\mathbb {Q}}
\newcommand{\ZZ}{\mathbb {Z}}
\newcommand{\RR}{\mathbb {R}}
\newcommand{\CC}{\mathbb {C}}

\newcommand{\FF}{\mathbb {F}}

\begin{document}

\title{On generalized projective product spaces and Dold manifolds}


\author{S. Sarkar and  P. Zvengrowski}

\date{\today}

\maketitle

\begin{abstract}
D. Davis introduced projective product spaces in 2010 as  a generalization of real projective spaces and discussed some of their topological properties.  On the other hand, Dold manifolds were introduced by A. Dold in 1956 to study the generators of the non-oriented cobordism ring. Recently, in 2019, A. Nath and P. Sankaran made a modest generalization of Dold manifolds. In this paper we simultaneously generalize both the notions of projective product spaces and Dold manifolds, leading to infinitely many different classes of new smooth manifolds. Our main goal will be to study the integral homology groups. cohomology rings, stable tangent bundles, and vector field problems, on certain generalized projective product spaces and Dold manifolds. 
\end{abstract}


\section{Introduction} \label{sec:1}
Real projective space $\RR P^m$ is the orbit space of the antipodal $\ZZ_2$-action on the standard $m$-dimensional sphere $S^m$. Extending this concept, D. Davis   introduced projective product spaces in  \cite{Davis} and studied several topological properties of these spaces. In fact, if $\ZZ_2$ acts on each $S^{m_1},  \ldots, S^{m_k}$ antipodally then  a projective product space is the orbit space of the diagonal $\ZZ_2$-action on the product $S^{m_1} \times \cdots \times S^{m_k} $, for some non-negative integers $m_1, \ldots, m_k$. 

 On the other hand,  A. Dold  in 1956  considered the diagonal $\ZZ_2$-action on the space $S^m \times \CC P^n$ where $\ZZ_2$ acts on $S^m$ antipodally and on $\CC P^n$ by complex conjugation. The orbit space  $(S^m \times \CC P^n)/\ZZ_2$ is now known as a Dold manifold, denoted by $D(m, n)$. These manifolds were introduced to study the generators of the non-oriented cobordism ring \cite{Dold}. Since then several interesting properties of Dold manifolds have been studied, see \cite{KZ}, \cite{Sankaran}, \cite{Nov} and \cite{Korbas} for a few of them. Recently, in  \cite{NS}, Nath and Sankaran made a slight generalization of Dold manifolds. 

Motivated by the above two concepts, we consider the following. Let $M$ be an $m$-dimensional manifold equipped with a free $\ZZ_2$-action and  $N$ an $n$-dimensional $\ZZ_2$-manifold. Then the diagonal $\ZZ_2$-action on the product $M \times N$ is free. So the orbit space $(M \times N)/ \ZZ_2$ is an $(m+n)$-dimensional manifold.  We call this manifold a {\it generalized projective product space} and denote it by $P(M,N)$. Note that Dold manifolds \cite{Dold},  projective product spaces \cite{Davis}, and the generalized Dold manifolds \cite{NS}, are all examples of this class of manifolds,  details are given below in Section \ref{sec:example_proj_prod_spaces}. The  main goal of this paper is to study several topological properties like (co)homologies, (stable) tangent bundles and (stable) spans of certain generalized projective product spaces.

We start in Section \ref{sec:toric_small} by  recalling the definition of toric manifolds, small covers, (real) moment angle manifolds and some relations among them, and  discuss some natural $\ZZ_2$-action on them. 

In Section \ref{sec:example_proj_prod_spaces}, we exhibit some interesting examples of generalized projective product spaces which were not previously studied. One may find several other interesting examples of generalized projective product spaces. However, we will be interested in studying topological properties of the manifolds in Examples \ref{ex:proj_prod_sp}, \ref{ex:toric} and \ref{ex:small}. 

In Section \ref{sec:cohomology}, we calculate the cohomology groups with integer coefficients and describe the cohomology ring with $\ZZ_2$ coefficients of certain generalized projective product spaces defined in Examples \ref{ex:proj_prod_sp}, \ref{ex:toric} and \ref{ex:small}.

In Section \ref{sec:tangent}, we describe canonical line bundles, study the (stable) tangent bundle of these spaces, and  compute the total Stiefel-Whitney characteristic classes of the manifolds defined in Examples \ref{ex:proj_prod_sp}, \ref{ex:toric} and \ref{ex:small}.
 
 In Section \ref{sec:span} we recall the definition of the span, denoted by ${\rm sp}(X)$, of a smooth manifold $X$.
 The orbit map $M \times N \to P(M, N)$ is a double cover. So  we have 
 \begin{equation}
 \mbox{sp}(P(M, N)) \geq \mbox{sp}(M/ \mathbb{Z}_2)
 \end{equation} 
by \cite[Theorem 1.7]{Sankaran}.
  We improve this lower bound for a wide class of the manifolds in  Examples \ref{ex:proj_prod_sp} and \ref{ex:toric}.     
We also study stable parallelizability of some of these manifolds.


\section{Toric manifolds and small covers}\label{sec:toric_small}
Toric manifolds and small covers were introduced and studied by M. W. Davis and T. Januszkiewicz in their pioneering paper \cite{DJ}. These categories of manifolds are topological generalizations of smooth projective toric varieties and real toric varieties, respectively. In this section, we recall the definition of these manifolds following \cite{DJ}. We also recall the definition of (real) moment angle manifolds equipped with a free $\ZZ_2$ action. We use these manifolds to construct infinitely many generalized projective product spaces in  Section \ref{sec:example_proj_prod_spaces}.

An $n$-dimensional simple polytope in $\RR^n$ is a convex polytope where exactly $n$ bounding hyperplanes meet at each vertex. For example, the $n$-simplex, the $n$-cube, and their finite Cartesian products are simple polytopes. Let $Q$ be a simple polytope. Then zero dimensional faces of $Q$ are called vertices, denoted by $V(Q)$, and codimension one faces of $Q$ are called facets, denoted by $\mathcal{F}(Q)$. Let $F(1):= \mathbb{R}$, $F(2) :=\mathbb{C}$, $T(1) :=\mathbb{Z}_2=\{x \in F(1) : |x|=1\}$ and $T(2):=S^1 = \{z \in F(2) : |z|=1\}$.
\begin{defn}
Let $j=1$ or $j=2$. A smooth action of $T(j)^n$ on a $jn$-dimensional smooth manifold $N^{jn}$ is said to be locally standard if every point $y \in N^{jn} $ has a $T(j)^n$-invariant open neighbourhood $U_y$ and a diffeomorphism $\psi_y \colon U_y \to V$, where $V$ is a $T(j)^n$-invariant open subset of $F(j)^n$, and an isomorphism $\delta_y \colon T(j)^n \to T(j)^n$ such that  $\psi_y (t\cdot x) = \delta_y (t) \cdot \psi_y(x)$ for all $(t,x) \in T(j)^n \times U_y$.
\end{defn}

We recall that such a map $\psi_y$ is known as a weakly equivariant map and, in addition, if $\delta_y$ is identity then it is called an equivariant map, or also a $T(j)^n$-equivariant map to emphasize the group action.

\begin{defn}\label{qtd02}
Let $j=1$ or $j=2$. A closed smooth $jn$-dimensional $T(j)^n$-manifold $N^{jn}$ is called a $T(j)^n$-manifold over a simple polytope $Q$ if the following conditions are satisfied:
\begin{enumerate}
\item the $T(j)^n$ action is locally standard,
\item the orbit map $\mathfrak{q}_j \colon N^{jn} \to Q$ sends  an $\ell$-dimensional orbit to a point in the interior of an $\ell$-dimensional face of $Q$.
\end{enumerate}
\end{defn}

We note that the manifold in  Definition \ref{qtd02} is known as a small cover when $j=1$ and is known as a toric manifold when $j=2$. They are different classes of manifolds, although in \cite{DJ} they are nicely combined to study some of their topological properties. But here we will deal with them separately.

\begin{example}
\begin{enumerate}
\item All complex projective spaces and their finite products are toric manifolds. 

\item All real projective spaces and their finite products are small covers. 
\end{enumerate}
\end{example}

Now we recall two different types of manifolds which are also central objects in toric topology. One is called the moment angle manifold and  the other is called the real moment angle manifold. We shall follow \cite[Subsection 4.1]{DJ} and show how they are related to the toric manifold and to the small cover, respectively. Let $F_1, \ldots, F_\mu$ be the facets of an $n$-dimensional simple polytope $Q$ and $G_i$ be the $i$th factor of $T(j)^{\mu}$ for $i=1, \ldots, \mu$ and $j=1, 2$. If $F$ is  a proper face of $Q$ of codimension-$k$, then $F = F_{i_1} \cap \cdots \cap F_{i_r}$ for a unique collection of facets $F_{i_1}, \ldots, F_{i_r}$. Let $T_F(j)$ be the subgroup of $T(j)^{\mu}$ generated by $\{ G_{i_1}, \ldots, G_{i_r}\}$. We fix $T_{Q}(j) = \{1\} \in T(j)^{\mu}$ for $j=1, 2$.
We define an equivalence relation $\sim$ on the product $T(j)^{\mu} \times Q$ as follows,
\begin{equation}\label{defeqiv}
(s, x) \sim (t, y) ~ \mbox{if and only if}~  x = y ~ \mbox{and} ~ ts^{-1} \in T_{F}(j)
\end{equation}
where $ F \subseteq Q $ is the unique face containing the point $ x $ in its relative interior. Then the identification space $$Z_Q(j)  := (T(j)^{\mu} \times Q)/\sim $$ is a manifold. This is called a real moment angle manifold if $j=1$ and called a moment angle manifold if $j=2$. So $Z_Q(j)$ is a $T(j)^{\mu}$-space for $j=1, 2$. We refer to \cite[Section 6]{BP} for a different construction and for further properties of moment angle complexes. Notice that  $\ZZ_2$ acts freely on $Z_Q(j)$ via the subgroup generated by $(-1, \ldots, -1) \in T(j)^{\mu}$ for $j=1, 2$.  

Next, we discuss the relation between a small cover $N^n$ over $Q$ and the real moment angle manifold $Z_Q(1)$. Let  $\mathfrak{q}_1 \colon N^{n} \to Q$ be the orbit map of a $T(1)^n$-manifold $N^{n}$ and  $\mathcal{F}(Q) =\{ F_1, \ldots, F_\mu\}$  be the facets of $Q$. So the subset $\mathfrak{q}_1^{-1}(F_i) $ is fixed by a subgroup $Z_i^1 \cong \mathbb{Z}_2$ of $T(1)^n$ for $i=1, \ldots, \mu$. The  subgroup $Z^1_i$ is determined by a unique element $\lambda_i \in T(1)^n = \ZZ^n_2 $. The assignment 
\begin{equation}\label{eq:char_func1}
 F_i \to \lambda_i : = \lambda(F_i)
\end{equation}
 for $i=1, \ldots , \mu$ is known as a $\ZZ_2$-colouring on $Q$. This assignment induces the following short exact sequence 
\begin{equation}\label{eq:smcv_exact}
0 \to \ker(\Lambda) \to \ZZ^m_2 \xrightarrow{\Lambda} \ZZ_2^n \to 0. 
\end{equation} 
From the discussion in \cite[Subsection 4.1]{DJ} one gets the following. 
\begin{prop}\label{prop:moment_small}
$\ker(\Lambda)$ acts freely on $Z_Q(1)$ with  $Z_Q(1)/ \ker(\Lambda)~\cong~N^n$. 
\end{prop}
We note that any sphere and finite product of spheres are real moment angle manifolds. Also each $\ZZ_2$ subgroup of $T(1)^n$ gives an involution on $N^{n}$.
\begin{remark}\label{rem:involution_sm_cv}
 Consider the involution $\bar{\tau}(1)$ on $T(1)^m$ (or $T(1)^n $) defined by $g \mapsto -g$. Then $\bar{\tau}(1)$ induces an involution $\tau(1)$ on $Z_Q(1)$ (or  $N^{n}$). 
\end{remark}

For the rest of this section we discuss the relation between a toric manifold $N^{2n}$ over $Q$ and the moment angle manifold $Z_Q(2)$. Let  $\mathfrak{q}_2 \colon N^{2n} \to Q$ be the orbit map of a $T(2)^n$-manifold $N^{2n}$ and  $\mathcal{F}(Q) =\{ F_1, \ldots, F_\mu\}$ the facets of $Q$. So the subset $\mathfrak{q}^{-1}(F_i) $ is fixed by a circle subgroup $S_i^1 \subseteq T(2)^n$ for $i=1, \ldots, \mu$. The assignment
\begin{equation}\label{eq:char_func2}
 F_i \to S^1_i 
\end{equation}
 for $i=1, \ldots , \mu$ is known as the characteristic function of $N^{2n}$, see \cite[(5.4)]{BP}. Note that the circle subgroup $S^1_i$ is uniquely determined by an element $(\lambda_{i_1}, \ldots, \lambda_{i_n}) \in \ZZ^n$ up to sign, where $\ZZ^n$ is the integral lattice in the Lie algebra of $T(2)^n$. The assignment  $\Lambda \colon F_i \to \lambda_i=(\lambda_{i_1}, \ldots, \lambda_{i_n})$ 
 for $i=1, \ldots , \mu$ is also known as the characteristic function on $Q$.  Therefore, by Definition \ref{qtd02}, this assignment induces the following short exact sequence of Lie groups
\begin{equation}\label{eq:torus_exact}
0 \to \ker({\rm exp}\Lambda) \to T(2)^{\mu} \xrightarrow{{\rm exp}\Lambda} T(2)^n \to 0. 
\end{equation} 
\begin{prop}\cite[Proposition 6.5]{BP}\label{prop:moment_quasitoric}
$\ker({\rm exp}\Lambda)$ is an $(m-n)$-dimensional torus subgroup of $T(2)^{\mu}$ and it acts freely on $Z_Q(2)$ with $Z_Q(2)/ \ker({\rm exp}\Lambda) \cong N^{2n}$. 
\end{prop}
Throughout the paper, for simplicity, we let $T^{m-n}:=\ker({\rm exp}\Lambda)$.
We note that any odd-dimensional sphere and finite products of odd-dimensional spheres are moment angle manifolds. Each $\ZZ_2$ subgroup of $T(2)^n$ gives an involution on $N^{2n}$. In this paper we shall consider the involution on toric manifolds given in the following example.

\begin{example}\label{ex:conjugation_toric}
The complex conjugation on each component of $T(2)^{\mu}$ commutes with $\ker({\rm exp}\Lambda)$. So, this gives a conjugation $\tau(2)$ on the toric manifold $N^{2n}$.  In particular, if $N^{2n}$ is the complex projective space $\CC P^n$, then $\tau(2)$ is the complex conjugation on $\CC P^n$ defined by $[z_0 : \ldots : z_n] \to [\bar{z_0} : \ldots : \bar{z_n}].$ 

Moreover, $\tau(2)$ preserves the relation in \eqref{defeqiv}. Therefore, any $T(2)^n$-invariant subset of $N^{2n}$ is also invariant under the conjugation $\tau(2)$.  
\end{example}


\section{Some generalized projective product spaces}\label{sec:example_proj_prod_spaces}

In this section, we discuss some examples of generalized  projective product spaces and compare them with  projective product spaces and (generalized) Dold manifolds. Of course, one can construct many other classes of generalized  projective product spaces but our interests will focus on certain subclasses of the following examples.

\begin{example}\label{ex:proj_prod_sp}
Let $ S^{m_i}$ and $S^{n_j}$ be spheres for $i=1, \ldots, k $ and $j=1, \ldots, \ell$. Consider the antipodal $\ZZ_2$-action on $S^{m_i}$ for $i=1, \ldots, k$. 
 Then the diagonal $\ZZ_2$-action on $ S(m_1, \ldots, m_k) := S^{m_1} \times \cdots \times S^{m_k}$ is free.
 Consider the action $\sigma_j$ on the $n_j$-dimensional sphere
$$S^{n_j} =\{(y_1, \ldots, y_{n_j+1}) \in \RR^{n_j+1} ~\big |~ \sum_{s=1}^{n_j+1} y_{s}^2=1\}$$ defined by 
\begin{equation}\label{eq:multiple_reflection}
\sigma_j \colon (y_1, \ldots, y_{p_j}, y_{p_j+1}, \ldots, y_{n_j+1}) \mapsto (y_1, \ldots, y_{p_j}, -y_{p_j+1}, \ldots, -y_{n_j+1})
\end{equation}
for some $0 \leq p_j \leq n_j$. So $\ZZ_2$ acts on  $S(n_1, \ldots, n_\ell) := S^{n_1} \times \cdots \times S^{n_\ell}$ via $ \sigma_1 \times \cdots \times \sigma_{\ell}$. Thus the $\ZZ_2$-action on  $S(m_1, \ldots, m_k) \times S(n_1, \ldots, n_\ell)$ defined by 
\begin{equation}\label{eq:multiple_reflection2}
({\bf x}_1, \ldots, {\bf x}_{k}), ({\bf y_1}, \ldots, {\bf y_\ell}) \mapsto (-{\bf x}_1, \ldots, -{\bf x}_{k}), (\sigma_1({\bf y_1}), \ldots, \sigma_\ell({\bf y_\ell}))
\end{equation}
 is free. The  orbit space is a generalized projective product space. We denote it
by  $P(m_1, \ldots, m_k; (n_1, p_1), \ldots, (n_\ell, p_\ell))$ (or by  $P_{\overline{m}, (\overline{n}, \overline{p})}$ where $\overline{m}=(m_1, \ldots, m_k)$ and $(\overline{n}, \overline{p})= ((n_1, p_1), \ldots, (n_\ell, p_\ell))$).

In particular, if  $\overline{m}=(m_1, \ldots, m_k)$ and $(\overline{n}, \overline{0})= (n_1, 0), \ldots, (n_\ell, 0))$ then  the manifold $P_{\overline{m}, (\overline{n}, \overline{0})}$ is called a projective product space in \cite{Davis} where it is denoted by $P_{\overline{s}}$ with $\overline{s}= (m_1,\ldots,m_k, n_1,\ldots, n_\ell)$.

More generally,  let $\sigma$ be a free action on $S(m_1, \ldots, m_k)$ and $\tau$ a (not necessarily free) involution on $S(n_1, \ldots, n_\ell)$. Then one can define the generalized projective product space $P(S(m_1, \ldots, m_k), S(n_1, \ldots n_\ell))$ using $\sigma$ and $\tau$. 
\end{example}

\begin{example}\label{ex:toric}
Let $X^{2n}$ be a $2n$-dimensional toric manifold and $\tau(2)$ an involution on $X^{2n}$ as defined in Example \ref{ex:conjugation_toric}.
Then one can define a $\ZZ_2$-action on $S(m_1, \ldots, m_k) \times X^{2n}$ by $$({\bf x}_1, \ldots, {\bf x}_k , y) \mapsto ({-\bf x}_1, \ldots, -{\bf x}_k, \tau(2)(y)).$$ This action is free and the orbit space is a generalized projective product space.  We denote it by $P(S(m_1, \ldots, m_k), X^{2n})$. 
 Note that the orbit map
\begin{equation}\label{double2}
S(m_1, \ldots, m_k) \times X^{2n} \longrightarrow P(S(m_1, \ldots, m_k), X^{2n})
\end{equation}
 is a double covering. Also the projection  $S(m_1, \ldots, m_k) \times X^{2n} \to S(m_1, \ldots, m_k)$ induces a smooth fibre bundle: 
\begin{equation}\label{fiber2}
X^{2n} \longrightarrow P(S(m_1, \ldots, m_k), X^{2n}) \longrightarrow  P_{\overline m}
\end{equation}
where $P_{\overline m}$ is a projective product space of \cite{Davis} for $\overline{m} = (m_1, \ldots, m_k)$.

In particular, if $X^{2n}$ is $\CC P^n$ then $P(S^m, \CC P^n)$ is the classical Dold manifold of \cite{Dold} denoted by $D(m, n)$. 

If $X^{2n} = \CC P^{n_1} \times \ldots \times \CC P^{n_{\ell}}$, which is a toric manifold, then we denote the manifold $P(S(m_1, \ldots, m_k), X^{2n})$ by  $P_T(m_1, \dots, m_{k}; n_1, \ldots, n_{\ell})$. 

Moreover, if $M$ is a manifold with free $\ZZ_2$-action then $P(M, X^{2n}) = (M \times X^{2n})/\ZZ_2$ is a generalized projective product space. 
\end{example}

\begin{example}\label{ex:small}
Let $Y^{n}$ be an $n$-dimensional small cover with an involution $\tau(1)$ as defined in Remark \ref{rem:involution_sm_cv}. Then the $\ZZ_2$-action on $S(m_1, \ldots, m_k) \times Y^{n}$ defined by $({\bf x}_1, \ldots, {\bf x}_{k}, y) \mapsto (-{\bf x}_1, \ldots, -{\bf x}_{k}, \tau(1)(y))$ is free, so the  orbit space is a generalized projective product space, denoted by $P(S(m_1, \ldots, m_k), Y^{n})$. 
 Note that the orbit map
\begin{equation}\label{double_small_cover}
S(m_1, \ldots, m_k) \times Y^{n} \longrightarrow P(S(m_1, \ldots, m_k), Y^{n})
\end{equation}
 is a double covering. The projection $S(m_1, \ldots, m_k) \times Y^{n} \to S(m_1, \ldots, m_k)$ induces a smooth fibre bundle: 
\begin{equation}\label{eq:fiber_small_cover}
Y^{n} \longrightarrow P(S(m_1, \ldots, m_k), Y^{n}) \longrightarrow P_{\overline{m}},
\end{equation}
where $\overline{m} = (m_1, \ldots, m_k)$.  
Let $Y^{n} =  \RR P^{n_1} \times \cdots \times \RR P^{n_\ell}$, which is a small cover.
Then we denote the corresponding generalized projective product space by $P_S(m_1, \ldots, m_{k}; n_1, \ldots, n_{\ell})$. 
\end{example}

\begin{example}\label{ex:nath-sankaran}
A generalization of Dold manifolds was made by A. Nath and P. Sankaran in \cite{NS}. Let $X$ be a smooth manifold with an involution $\tau$. 
Then $\ZZ_2$ acts freely on $S^m\times X$ via the map $({\bf x},  y)\sim (-{\bf x}, \tau(y))$. They called the orbit space a 
 \textit{generalized~Dold~manifold},  denoted by $P(m, X)$, and studied several topological properties for various $X$. This is also the reason we call our manifolds generalized projective product spaces.  
\end{example}

One can construct many possibly interesting  generalized projective product spaces. Some further ones are listed below, but will not be studied in this paper.   

\begin{example}
Let $M_1, \ldots, M_k$ be (real) moment angle manifolds. Then $\ZZ_2$ acts on each $M_i$ freely for $i=1, \ldots, k$ by Section \ref{sec:toric_small}. Thus $\ZZ_2$ acts freely on $M_1 \times \cdots \times M_k$ via diagonal action. So the orbit space $(M_1 \times \cdots \times M_k)/\ZZ_2$ is a generalized projective product space.  We remark that a finite product of   (real) moment angle manifolds is again a  (real) moment angle manifold. So this example may look artificial, however, the projective product spaces of \cite{Davis} belong to this class of manifolds. 
\end{example}

\begin{example}
Let $M$ be a (real) moment angle manifold corresponding to a simple polytope and $N$ be an almost complex manifold with an involution, or a toric manifold, or a homogeneous space with an involution. Then $\ZZ_2$ acts freely on $M \times N$  where $\ZZ_2$ acts freely on $M$. The orbit space $(M \times N)/\ZZ_2$ is then a generalized projective product space. 
\end{example}

\begin{example}
Let $G$ be a subgroup of the general linear group of $GL(n, \RR)$ or of $GL(n, \CC)$, such that $\ZZ_2$ acts freely on $G$. Let $N$ be an almost complex manifold with an involution or a toric manifold or a homogeneous space with an involution. Then $\ZZ_2$ acts freely on $G \times N$, and the orbit space $(G \times N)/\ZZ_2$ is a generalized projective product space. 
\end{example}

\begin{example}
Let $G$ be a subgroup of the general linear group of $GL(n, \RR)$ or $GL(n, \CC)$ such that $\pi_1(G)=\ZZ_2$. Let $N$ be an almost complex manifold with an involution, or a toric manifold, or a homogeneous space with an involution. Then $\ZZ_2$ acts freely on $\widetilde{G} \times N$  where $\widetilde{G}$ is the universal cover of $G$. The orbit space $(G \times N)/\ZZ_2$ is a generalized projective product space.  For example, $\pi_1(SO(n)) = \ZZ_2$ if $n \geq 3$.
\end{example}


\section{Cohomology of manifolds in Section \ref{sec:example_proj_prod_spaces}}\label{sec:cohomology}

In this section, we first compute the cohomology ring with $\ZZ_2$ coefficients and $\mathbb{Q}$ coefficients of the manifolds in Example \ref{ex:proj_prod_sp}. Then  we construct a cell structure on the generalized projective product space $P(S^m,N)$, where $N$ is a toric manifold, and describe the cohomology of this space with $\ZZ$ and $\QQ$ coefficients. We also compute the cohomology of manifolds in Examples \ref{ex:toric} and \ref{ex:small} with $\ZZ_2$ coefficients.

\subsection{Cohomology of manifolds in Example \ref{ex:proj_prod_sp}}\label{subsec:proj_prod_sp}
 In \cite[Section 1]{Davis}, the author showed that $P(m_1, \ldots, m_k)$ is a sphere bundle over $P(m_1, \ldots, m_{k-1})$ for $k \geq 2$ where $P(m_1):=\RR P^{m_1}$. We show that the manifold  $P(m_1, \ldots, m_k; (n_1, p_1), \ldots, (n_\ell, p_\ell))$ defined in Example \ref{ex:proj_prod_sp} is an iterated sphere bundle over $P(m_1, \ldots, m_k)$. Then we compute its cohomology ring with $\ZZ_2$ and $\QQ$ coefficients. The calculation is similar to that of the cohomology with $\ZZ_2$ coefficients and $\mathbb{Q}$ coefficients of projective product spaces in \cite[Section 2]{Davis}. 
  We consider the trivial sphere bundle $$S(m_1, \ldots, m_k) \times S(n_1,  \ldots, n_{\ell-1}) \times S^{n_\ell} \xrightarrow{\tilde{\xi}}  S(m_1, \ldots, m_k) \times S(n_1, \ldots, n_{\ell -1}).$$  The group $\ZZ_2$ acts on its total space  by
   $$({\bf x}_1,.., {\bf x}_k, {\bf y}_1,.., {\bf y}_{\ell-1}, {\bf y}_\ell) \mapsto (-{\bf x}_1,.., -{\bf x}_k,\sigma_1({\bf y}_1),.., \sigma_{\ell-1}({\bf y}_{\ell-1}), \sigma_{\ell}({\bf y}_{\ell}))$$
and on its base by  $$({\bf x}_1, \ldots, {\bf x}_k, {\bf y}_1, \ldots, {\bf y}_{\ell-1}) \mapsto (-{\bf x}_1,\ldots, -{\bf x}_k,\sigma_1({\bf y}_1), \ldots, \sigma_{\ell-1}({\bf y}_{\ell-1}))$$  where the actions $\sigma_1, \ldots, \sigma_{\ell-1}$ and $\sigma_{\ell}$ are defined in \eqref{eq:multiple_reflection}. So $\tilde{\xi}$ is a $\ZZ_2$-equivariant map. Since $\ZZ_2$ acts freely on the base, $\tilde{\xi}$ induces a sphere bundle  
\begin{equation}\label{eq:sphere_bundle}
P(m_1,..., m_k; (n_1, p_1),..., (n_\ell, p_\ell)) \to P(m_1,..., m_k; (n_1, p_1),..., (n_{\ell-1}, p_{\ell-1}))          
\end{equation} 
for any $\ell \geq 2$. By similar arguments we can show that $$P(m_1, \ldots, m_k; (n_1, p_1)) \to P(m_1, \ldots, m_k) $$ is a sphere bundle with fibre $S^{n_1}$. For simplicity of notation we let this bundle correspond to $\ell =1$. Next,  we consider the trivial line bundle $$S(m_1, \ldots, m_k) \times S(n_1, \ldots, n_{\ell -1}) \times \RR \xrightarrow{}  S(m_1, \ldots, m_k) \times S(n_1, \ldots, n_{\ell -1}).$$  The group $\ZZ_2$ acts on its total space  by $$({\bf x}_1, \ldots, {\bf x}_k, {\bf y}_1, \ldots, {\bf y}_{\ell-1}, r) \mapsto (-{\bf x}_1,\ldots, -{\bf x}_k,\sigma_1({\bf y}_1), \ldots, \sigma_{\ell-1}({\bf y}_{\ell-1}), -r )).$$ This induces a line bundle 
\begin{equation}\label{eq:can_line_bundle}
\eta_\ell \colon E  \to  P(m_1, \ldots, m_k; (n_1, p_1), \ldots, (n_{\ell-1}, p_{\ell-1})) .
\end{equation} 
 So the bundle $E(n_\ell, p_\ell) : = p_\ell \epsilon \oplus (n_\ell - p_\ell +1)\eta_\ell$ is a vector bundle with fibre  $\RR^{(n_\ell +1)}$  over $P(m_1, \ldots, m_k; (n_1, p_1), \ldots, (n_{\ell-1}, p_{\ell-1}))$ where $\epsilon$ is the trivial line bundle. From the action $\sigma_\ell$ in \eqref{eq:multiple_reflection}, we can conclude that the sphere bundle $$S(p_\ell \epsilon \oplus (n_\ell - p_\ell +1)\eta_\ell) \to  P(m_1, \ldots, m_k; (n_1, p_1), \ldots, (n_{\ell-1}, p_{\ell-1}))$$ is the sphere bundle in \eqref{eq:sphere_bundle}. We denote the associated disk bundle by $D(E(n_{\ell}, p_\ell))$ (or by $D(p_\ell \epsilon \oplus (n_\ell - p_\ell +1)\eta_\ell)$), for any $\ell \geq 1$.

We denote an exterior algebra over $\ZZ_2$ by $\Lambda(-)$ and the total Steenrod square by  $Sq = \sum_{n \geq 0} Sq^n$. If $\alpha \in H^q(X)$, then we write $|\alpha| = q$ for the degree of $\alpha$. 
\begin{thm}\label{thm_cohom_gen_proj_prod}
Let $m_1\leq \cdots \leq m_k\leq n_1 \cdots \leq n_\ell$, $m_1< m_k$ (or $m_1$ is odd) and $p_j > 1, 1 \leq j \leq \ell$. Then  $H^*(P(m_1, \ldots, m_k; (n_1, p_1), \ldots, (n_\ell, p_\ell)); \ZZ_2) $ is isomorphic as a graded $\ZZ_2$-algebra to  
\begin{align*}
\ZZ_2[\alpha]/(\alpha^{m_1+1})\otimes\Lambda(\alpha_2, \ldots, \alpha_k)\otimes \Lambda(\beta_1,\ldots, \beta_\ell),
\end{align*} where $|\alpha|=1$, $|\alpha_i|=m_i,$ for $2 \leq i \leq k $ and $|\beta_j|=n_j,$ for $1 \leq j \leq \ell,$ $Sq(\alpha_i)=(1+\alpha)^{m_i+1}\alpha_i,$ $Sq(\beta_j)=(1+\alpha)^{n_j+1-p_j} \beta_j$ and $Sq(\alpha) = \alpha(1+\alpha)$.
 
If $p_j=1$ for some $j$, then $\beta_j^2=\alpha^{n_j} \beta_j $ for $n_j=m_1.$ If in addition, $m_1$ is even, then $\alpha_i^2=\alpha^{m_i}\alpha_i$ for all  $i \geq2$ whenever $m_i=m_1.$
\end{thm}
\begin{proof}
This is similar to the proof of \cite[Theorem 2.1]{Davis}, but some clarifications are needed. The author in 
\cite{Davis} shows that $$H^*(P(m_1, \ldots, m_k); \ZZ_2) \cong \ZZ_2[\alpha]/(\alpha^{m_1+1})\otimes\Lambda(\alpha_2, \ldots, \alpha_k)$$ as a graded $\ZZ_2$-algebra,  and the corresponding relations among $\alpha, \alpha_2, \ldots, \alpha_k$ holds. So the result is true for $\ell=0$, which starts the proof by induction. 
We prove the claim when $\ell =1$ that is for $H^*(P(m_1, \ldots, m_k; (n_1, p_1)); \ZZ_2)$. Then by induction one can complete the proof.
We have shown the isomorphism $P(m_1, \ldots, m_k; (n_1, p_1)) \cong S(p_1 \epsilon \oplus (n_1+1-p_1)\eta_1)$ as a sphere bundle over $P(m_1, \ldots, m_k)$. This gives the cofibration $$P(m_1, \ldots, m_k; (n_1, p_1)) \xrightarrow{q} P(m_1, \ldots, m_k) \xrightarrow{\iota} T(E(n_1,p_1)) \simeq M(q)$$
where the first map is given by $q([{\bf x}_1, \ldots, {\bf x}_k, {\bf y}_1]) =[{\bf x}_1, \ldots, {\bf x}_k]$, $T(E(n_1, p_1))$ denotes the Thom space of the bundle $E(n_1, p_1)$ defined after \eqref{eq:can_line_bundle}, and $M(q)$ is the mapping cone of $q$. Hence one gets the following long exact sequence with coefficients in $\ZZ_2$.
$$ \cdots \to H^*(T(E(n_1, p_1)))  \xrightarrow{\iota^*}  H^*(P_{\overline m})  \xrightarrow{q^*}  H^*(P_{\overline m, (\overline n, \overline p)}) \xrightarrow{\delta} H^{*+1}(T(E(n_1, p_1))) \to \cdots $$ where $P_{\overline{m}} = P(m_1, \ldots, m_k)$ and $P_{\overline m, (\overline n, \overline p) }= P(m_1, \ldots, m_k; (n_1, p_1))$.
Since our assumption is $m_1 \leq \cdots \leq m_k \leq n_1$, we have the map $\phi \colon   P_{\overline m} \to P_{\overline m, (\overline n, \overline p)}$ defined by  $[{\bf x}_1, \ldots, {\bf x}_k] \mapsto [{\bf x}_1, \ldots, {\bf x}_k, {\bf x}_k]$. So the composition map $ q \circ \phi$ gives the identity on $ P_{\overline m}$. Thus we get the splitting $$H^*( P_{\overline m, (\overline n, \overline p)}; \ZZ_2) \approx H^*( P_{\overline m }; \ZZ_2) \oplus H^{*+1}(T(E(n_1, p_1)) ;\ZZ_2).$$  Let $\beta_1$ be the image of the Thom class in $H^{n_1+1}(T(E(n_1,p_1)))$ under this isomorphism. Note that the Thom isomorphism  gives $H^{*+1}(T(E(n_1,p_1)); \ZZ_2 ) \approx H^*(P_{\overline m} ; \ZZ_2) \cdot \beta_1 $.  Then one can say that  $$H^*(P_{\overline m, (\overline n, \overline p)}); \ZZ_2) \approx H^*(P_{\overline m}; \ZZ_2) \oplus H^*(P_{\overline m} ; \ZZ_2) \cdot \beta_1.$$

The projection $pr \colon P_{\overline m, (\overline n, \overline p)} \to P(m_1) \approx \RR P^{m_1}$ gives $pr^*(\eta_0)=\eta_1$ by naturality, where $\eta_0$ is the canonical line bundle on $\RR P^{m_1}$ and $\eta_{\ell}$ is defined in \eqref{eq:can_line_bundle}. Therefore the total Steenrod square and the total Stiefel-Whitney class have the following relation in our setting, using the arguments in \cite[Page 94]{MiSt}. 
\begin{align*}
Sq(\beta_1)&=W(p_1\epsilon \oplus (n_1+1-p_1) \eta_1) \beta_1 \\ &= (1+\alpha)^{n_1+1-p_1}\beta_1
\end{align*} 
where $\alpha = w_1(\eta_0)$, the canonical generator of $H^1(\RR P^{m_1}; \ZZ_2)$, and $|\beta_1| = n_1$. Then $\beta_1^2 = \binom{n_1+1-p_1}{n_1}\alpha^{n_1} \beta_1$ and hence $\beta_1^2$ is zero for $p_1>1$ and the theorem holds in this case.
For $p_1=1$, we get the same structure with $\beta_1^2=\alpha^{n_1} \beta_1.$
\end{proof}

We recall that if $\mathbb{F} =\mathbb{Q}$, or $\ZZ_p$ for an odd positive prime $p$, then $$H^*(S(m_1, \ldots, m_k) \times S(n_1, \ldots, n_\ell); \mathbb{F}) \approx \frac{\mathbb{F}[\delta_1, \ldots, \delta_k, \gamma_1, \ldots, \gamma_\ell]}{(\delta_i^2, \gamma_j^2 ~:~ 1 \leq i \leq k, 1 \leq j \leq \ell )}$$ with $|\delta_i| = m_i$ and $|\gamma_j|=n_j$ for  $1 \leq i \leq k, 1 \leq j \leq \ell$. The proof of the next proposition is very similar to \cite[Theorem 2.8]{Davis}, so we omit the details. 

\begin{prop}
Let $ \prod_{i=1}^k S^{m_i} \times \prod_{j=1}^{\ell}S^{n_j} \xrightarrow{q} P_{\overline m, (\overline n, \overline p)}$  be the orbit map of the $\ZZ_2$-action and  $\mathbb{F} =\mathbb{Q}$, or $\ZZ_p$ for an odd positive prime $p$, with $\overline{m}=(m_1, \ldots, m_k) $ and $ (\overline n, \overline p) = ((n_1, p_1), \ldots, (n_\ell, p_\ell))$. Then the image of the map $$ H^*(P_{\overline m, (\overline n, \overline p)}; \FF) \xrightarrow{q^*} H^*(\displaystyle S(m_1, \ldots, m_k) \times S(n_1, \ldots, n_{\ell}) ; \FF)$$ is the $\FF$-span of the products $\delta_{i_1} \cdots \delta_{i_u} \cdot \gamma_{j_1} \cdots \gamma_{j_v}$ such that $\sum_{r=1}^{u}(m_{i_r}+1) + \sum_{s=1}^{v} (n_{j_s} -p_{j_s}+1) $ is even. 
\end{prop}

We note that the class of projective product spaces is a proper subset of the class of manifolds defined in Example \ref{ex:proj_prod_sp}, but one can extend most of the results of the paper \cite{Davis} to this new class of manifolds.  


\subsection{Cohomology of some manifolds in Example \ref{ex:toric}}\label{subsec_cohom_toric}
First we recall an invariant cell structure and describe the cohomology ring of a toric manifold following \cite{DJ}. Then we give a cell structure on $P(S^m, X)$ where $X$ is a $2n$-dimensional toric manifold. We shall compute the integral homology and rational cohomology of these spaces. 

Let $Q$ be an $n$-dimensional simple polytope and  $f \colon  Q \subset \RR^n \to \RR$  the restriction of a linear map which distinguishes the vertices $V(Q)$ of $Q$. This induces an ordering on the vertices  of $Q$, and consequently an orientation on each edge of $Q$ so that $f$ is an increasing map along it.  Thus the union of all edges $E(Q) \subset Q$ forms a directed graph, see Figure \ref{Fig_prism} for an example. The index of $v \in V(Q)$ is the number of edges in $E(Q)$ orienting towards $v$. For example, in Figure \ref{Fig_prism}, the index of $v_2$ is one. Let $h_i$ be the number of vertices of index $i$ for $0 \leq i \leq n$. Let $F_v$ be the maximal face of $Q$ containing only the inward edges at $v$ and $U_v$ be the open subset of $F_v$ obtained by deleting all the faces of $F_v$ not containing the vertex $v$. Then $Q = \sqcup_{v} U_v$. 
\begin{figure}
\begin{center}

\begin{tikzpicture}
\draw[dashed] (0,0)--(4,1)--(6,3)--(2,2)--cycle;
\draw (0,0)--(2,2)--(3,-1)--cycle;
\draw[dashed] (4,1)--(6,3)--(7,0)--cycle;
\draw (6,3)--(2,2)--(3,-1)--(7,0)--cycle;

\draw[->] (1.9,.47)--(2,.5);
\draw[->] (4.9,1.9)--(5,2);
\draw[->] (1.6,-.53)--(1.5,-.5);
\draw[->] (.9,.9)--(1,1);
\draw[->] (3.9,2.47)--(4,2.5);
\draw[->] (5,-.5)--(5.2,-.47);
\draw[->] (5.6,.47)--(5.5,.5);
\draw[->] (2.51,.48)--(2.5,.5);
\draw[->] (6.51,1.48)--(6.5,1.5);
\node [left] at (0,0) {$v_1$};
\node [left] at (2,2) {$v_3$};
\node [left] at (3.3,-1.2) {$v_0$};
\node [left] at (4.1,1.2) {$v_4$};
\node [right] at (7,0) {$v_2$};
\node [right] at (6,3) {$v_4$};

\end{tikzpicture}
\end{center}
\caption{An orientation on the edges of a prism.}
\label{Fig_prism}
\end{figure}
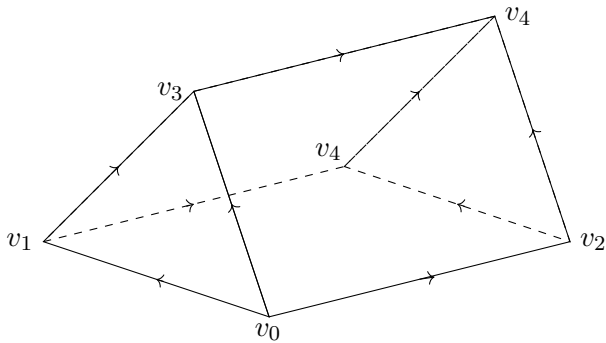

Let $X$ be a $2n$-dimensional toric manifold and $\mathfrak{q} \colon X \to Q$ the orbit map as in Definition \ref{qtd02}. Then $\widetilde{U}_v := \mathfrak{q}^{-1}(U_v)$ is a torus invariant subset of $X$ and it is weakly equivariantly homeomorphic to $\CC^{\dim(F_v)}$, where the torus action is standard. So  $\widetilde{U}_v$ is  $2 \dim(F_v)$-dimensional and invariant under the involution considered in Example \ref{ex:conjugation_toric}. Therefore, $X = \sqcup_{v}  \widetilde{U}_v$ gives an invariant cell structure where there is no odd-dimensional cell. The authors in \cite{DJ} showed that number of $2 i$-dimensional cells in $X$ is $h_i$ for $0 \leq i \leq n$.

Let  $\mathcal{F}(Q):= \{F_1,\ldots,F_{\mu}\}$ be the facets of $Q$.
Recall the assignment  $\lambda(F_i) =\lambda_i :=(\lambda_{i_1}, \ldots, \lambda_{i_n}) \in \ZZ^n$ for each  $i \in \{1, \ldots, {\mu}\}$ determined by \eqref{eq:char_func2}. Let $I$ and $J$ be the ideals of $\ZZ[u_1,\ldots, u_{\mu}]$ generated by the sets 
\begin{equation}\label{eq:ideal_I-J}
\{u_{j_1}\cdots u_{j_k}\colon \cap_{i=1}^k F_{j_i} = \phi\} ~~\mbox{and} ~~ \{\lambda_{1_\ell}u_1+ \cdots + \lambda_{{\mu}_\ell}u_{\mu} \colon 1\leq \ell \leq n\}
\end{equation}
 respectively. The ideal $I$ is known as Stanley-Reisner ideal. Then by  \cite[Theorem 4.14]{DJ} we have the following.
\begin{equation}\label{eq:cohom_toric_mfds}
H^*(X; \ZZ) \cong \frac{\ZZ [u_1, \ldots, u_{\mu}]}{I+J}
\end{equation}\\
where $u_i$ is the Poincar\'e dual of $\mathfrak{q}^{-1}(F_i)$ for $i=1, \ldots,  {\mu}$.

Let $B_i^{+} = \{(x_1, \ldots, x_i, 0, \ldots 0) \in S^m ~|~ x_i > 0 \}$ and   $B_i^{-} = \{(x_1, \ldots, x_i, 0, \ldots 0) \in S^m ~|~ x_i < 0 \}$. So, $\{B_i^+, B_i^- ~|~ i=0, \ldots, m\}$ gives a cell structure on $S^m$ such that $B_i^+ \to B_i^- $ and $B_i^- \to B_i^+ $ are homeomorphism under the antipodal action on $S^m$. Therefore the collection 
\begin{equation}\label{eq:cell_prod_1}
\{ B_i^+ \times \widetilde{U}_v, B_i^-  \times \widetilde{U}_v ~|~ i=0, \ldots, m ~\mbox{and} ~ v \in V(Q) \}
\end{equation}
 gives  a cell structure on $S^m\times X$ which is invariant under the $\ZZ_2$-action considered in Example \ref{ex:toric}.
So this induces a cell structure on $P(S^m, X)$.  More precisely, let $\phi \colon S^m \times X \to P(S^m, X)$ be the orbit map as in \eqref{double2} and we denote $(B_i, \widetilde{U}_v) := \phi(B_i^+  \times \widetilde{U}_v)$. Then   
\begin{equation}\label{eq:cell_prod_2}
 \{ (B_i, \widetilde{U}_v) ~|~ i=0, \ldots, m ~\mbox{and} ~ v \in V(Q) \}
\end{equation} 
gives a cell structure on $ P(S^m, X)$.  

Note that the boundary map for the chain complex determined by the cells in \eqref{eq:cell_prod_1} is given by the following.
$$\partial(B_0^{\pm} \times  \widetilde{U}_v  =0) ~\mbox{and}~ \partial(B_i^{\pm} \times  \widetilde{U}_v) = \pm (B_{i-1}^+ \times \widetilde{U}_v + B_{i-1}^- \times \widetilde{U}_v) $$ for $i=1, \ldots, m$ and $v \in V(Q)$. Let $-1$ be the non-trivial element in $\ZZ_2$. So $-1 \colon S^m \times X \to S^m \times X$ is a homeomorphism which preserves the above cell structure. So, on the chain complex, it induces the following $$-1(B_i^{\pm} \times \widetilde{U}_v) = (-1)^{i + \dim(U_v)+1} (B_i^{\mp} \times \widetilde{U}_v). $$ So the boundary map for the chain complex determined by \eqref{eq:cell_prod_2} is given by 
$$\partial(B_0 \times  \widetilde{U}_v  =0) ~\mbox{and}~ \partial(B_i \times  \widetilde{U}_v) = (1 + (-1)^{i+\dim(U_v)})  (B_{i-1} \times \widetilde{U}_v) $$ for $i=1, \ldots, m$ and $v \in V(Q)$. Let $b^iu^{\dim(U_v)}$ be the cochain dual to $(B_i \times  \widetilde{U}_v)$. So the corresponding coboundary map is given by  $$\delta(b^i u^{\dim(U_v)}) = (1 + (-1)^{i + \dim(U_v) +1}) b^{i+1} u^{\dim(U_v)}$$ for $i=0, \ldots, m$ and $v \in V(Q)$. Therefore one can obtain the following. 
\begin{prop}\label{prop:intg_cohom_toric_1}
The integral cohomology group of $P(S^m, X)$ is given by the following.
\begin{enumerate}
\item {\bf For $m$ even}. Abelian group generated by $u^{\dim(U_v)}$, $b^mu^{\dim(U_v)+1}$,  $b^{2i}u^{\dim(U_v)}$ and $b^{2i-1}u^{\dim(U_v)+1}$ where  $u^{\dim(U_v)}$, $b^mu^{\dim(U_v)+1}$ have infinite order and $b^{2i}u^{ \dim(U_v)}$, $b^{2i-1}u^{\dim(U_v)+1}$ have order $2$. 

\item {\bf For $m$ odd}. Abelian group generated by $u^{\dim(U_v)},$ $b^mu^{\dim(U_v)}$,  $b^{2i}u^{\dim(U_v)}$ and $b^{2i-1}u^{\dim(U_v)+1}$ where  $u^{\dim(U_v)}$, $b^mu^{\dim(U_v)+1}$ have infinite order and $b^{2i}u^{ \dim(U_v)}$, $b^{2i-1}u^{\dim(U_v)+1}$ have order $2$. 
\end{enumerate}
where $i=1, \ldots, [m/2]$ and $0 \leq \dim(U_v) \leq n$ with $\dim(U_v)$ is even.
\end{prop}

We note that $\CC P^n$ is an example of toric manifold. So, Proposition \ref{prop:intg_cohom_toric_1} generalizes \cite[Proposition 1.6]{Fuj}. Even though the result may seem similar but several crucial facts about toric manifolds are needed here. 

Next proposition says the cohomology of $P(m_1, \ldots, m_k; X)$ with coefficients in  $\mathbb{F} =\mathbb{Q}$ or $\ZZ_p$ for an odd positive prime $p$.
\begin{prop}
Let $S(m_1, \ldots, m_k) \times X \xrightarrow{g} P(m_1, \ldots, m_k; X)$ be the orbit map as in \eqref{double2} and  $\mathbb{F} =\mathbb{Q}$ or $\ZZ_p$ for an odd positive prime $p$. Then $$ H^*(P(m_1, \ldots, m_k; X); \FF) \cong A \otimes  H^* (X; \FF)$$
where $A$ is the $\FF$-span of the products $\delta_{i_1} \cdots \delta_{i_u} $ in $H^* (S(m_1, \ldots, m_k); \FF)$  such that $\sum_{r=1}^{u}(m_{i_r}+1) $ is even. 
\end{prop}
\begin{proof}
The map $g$ is a double cover. So there is a map $P(m_1,  \ldots, m_k; X) \to \mathbb{R} P^{\infty}$ giving the fibration $S(m_1, \ldots, m_k) \times X \xrightarrow{g}  P(m_1,  \ldots, m_k; X) \to \mathbb{R} P^{\infty}$. So by Serre spectral sequence with local coefficients one has the following $E_2$-page in the first quadrant: $E_2^{r, s} = H^r(\mathbb{R} P^{\infty}; \mathcal{H}^s (S(m_1, \ldots, m_k) \times X; \FF))$
which converges to $H^*(P(m_1,  \ldots, m_k; X)  ) $. As $\FF$ is a field of characteristic not equal to $2$, one has 
\begin{equation*}
H^r(\mathbb{R} P^{\infty}; \mathcal{H}^s (\prod_{i=1}^k S^{m_i} \times X; \FF)) = \left\{ \begin{array}{ll} (H^s (\prod_{i=1}^k S^{m_i} \times X; \FF)))^{\pi_1(\mathbb{R} P^{\infty})} & \mbox{if}~ r=0\\
 0  & \mbox{if}~ r \neq 0,
\end{array} \right.
\end{equation*}
where $\pi_1(\mathbb{R} P^{\infty}) =\ZZ_2$ action on $H^s (S(m_1, \ldots, m_k) \times X; \FF))$ is induced from the $\pi_1(\mathbb{R} P^{\infty}) $ action on the fibre. The non-trivial element of $\ZZ_2$ acts on $\delta_{i_1} \cdots \delta_{i_u} \in H^* (S(m_1, \ldots, m_k); \FF))$ by multiplication with the product $(-1)^{m_{i_1} +1} \cdots (-1)^{m_{i_u}+1}$. Since $\ZZ_2$-action on $X$ is locally isomorphic to the complex conjugation by Example \ref{ex:conjugation_toric},  the non-trivial element of $\ZZ_2$ acts on each generator of $H^*(X; \FF)$ trivially. Therefore the conclusion follows from the Kunneth theorem and degeneracy of the spectral sequence at $E_2$-page. 
\end{proof}

In the rest of this subsection, we compute the cohomology ring of the generalized projective product space $P(m_1, \ldots, m_k; X)$ with $\ZZ_2$ coefficients. 
\begin{thm}\label{thm:toric_mod2}
Let $m_1, \ldots, m_k$ are positive integer greater than one. Then
$H^*(P(m_1, \ldots, m_k; X)); \ZZ_2) \cong H^*(P(m_1, \ldots, m_k);\ZZ_2) \otimes H^*(X; \ZZ_2).$
\end{thm}  
\begin{proof}
Recall the fibre bundle in \eqref{fiber2}. By hypothesis $\pi_1(P(m_1, \ldots, m_k)) = \ZZ_2$. From the definition of $P(m_1, \ldots, m_k; X)$, we get that the group $\ZZ_2$ acts on the fibre $X$ by complex conjugation locally. Hence it acts on $H^*(X; \ZZ_2)$ trivially. Note that the cohomology groups of the fibre and base of this bundle has finite dimension over the field $\ZZ_2$. Also both the fibre $X$ and the base $P(m_1, \ldots, m_k)$ are path connected. By \cite[Theorem 3.1]{DJ}, $H^*(X; \ZZ_2)$ is concentrated in even degrees with $H^{2i}(X; \ZZ_2)= \ZZ_2^{h_i}$ for $i=0, \ldots, n$. So, by applying \cite[Proposition 5.5]{Mcc}  one gets that the corresponding spectral sequence collapses at $E_2$. So $X$ is totally non-homologous to zero in $P(m_1, \ldots, m_k; X)$ with respect to $\ZZ_2$. Thus by \cite[Theorem 5.10]{Mcc}, one gets the result. 
\end{proof}

\subsection{Cohomology of some manifolds in Example \ref{ex:small}}
Let $Y$ be a small cover over an $n$-dimensional simple polytope $Q$ and $\tau(1)$ an involution on $Y$ as in Remark \ref{rem:involution_sm_cv}. 
In this subsection, we compute the cohomology ring of the generalized projective product space $P(m_1, \ldots, m_k; Y)$ with $\ZZ_2$ coefficients. 
We remark that this can be computed similarly as in Theorem \ref{thm:toric_mod2}. For completeness we write the arguments briefly. 

Let $\mathfrak{q} \colon Y \to Q $ be the orbit map corresponding to the small cover $Y$ as in Definition \ref{qtd02}. The set $\widetilde{U}_v := \mathfrak{q}^{-1}(U_v)$ is a $\ZZ_2^n$-invariant subset of $Y$ and it is weakly equivariantly homeomorphic to $\RR^{\dim(F_v)}$ with respect to $\ZZ_2^{\dim(F_v)}$-action. So  $\widetilde{U}_v$ is  a $\dim(F_v)$-dimensional subset of $Y$ and invariant under the involution considered in  Remark \ref{rem:involution_sm_cv}. Therefore, $Y = \sqcup_{v}  \widetilde{U}_v$ gives an invariant cell structure . By  \cite[Theorem 4.14]{DJ} we have the following.
\begin{equation}\label{eq:cohom_smcv}
H^*(Y; \ZZ_2) \cong \frac{\ZZ_2 [u_1, \ldots, u_{\mu}]}{I+J}
\end{equation}\\
where the ideal $I, J$ are defined similarly as in \eqref{eq:ideal_I-J} with  $\lambda_i :=(\lambda_{i_1}, \ldots, \lambda_{i_n}) \in \ZZ_2^n,$ for each  $i \in \{1, \ldots, {\mu}\}.$  
Proof of the next proposition is similar to the proof of Theorem \ref{thm:toric_mod2} so we omit the details. We note here  $H^{i}(Y; \ZZ_2)= \ZZ_2^{h_i}$ for $i=0, \ldots, n$.
\begin{prop}\label{prop:smcv_mod2}
Let $m_1, \ldots, m_k$ are positive integer greater than one. Then
$H^*(P(m_1, \ldots, m_k; Y)); \ZZ_2) \cong H^*(P(m_1, \ldots, m_k);\ZZ_2) \otimes H^*(Y; \ZZ_2).$
\end{prop}  

We remark that it might be interesting to ask the torsion in the integral cohomology of the manifolds in Example \ref{ex:small}.


\section{Tangent Bundles on spaces in Section \ref{sec:example_proj_prod_spaces}}\label{sec:tangent}
In this section, we study some natural bundles on the generalized projective product spaces considered in Section  \ref{sec:example_proj_prod_spaces}. Then we compute their Stiefel-Whitney characteristic classes. Moreover, we give a nice upper bound for the immersion dimension of the manifolds in Example \ref{ex:proj_prod_sp}. The immersion dimension of a manifold $M$, denoted by $\mbox{imm}(M)$, is the smallest $d$ such that there is an immersion $M \to \mathbb{R}^d$.  

\subsection{Stable tangent bundle on some manifolds in Example \ref{ex:proj_prod_sp}}
We adhere the notation from  Example \ref{ex:proj_prod_sp}. 
The tangent bundle on the product $ S(m_1, \ldots, m_k) \times  S(n_1, \ldots, n_{\ell})$ is given by
 $$ \{( \bar{\bf x}, \bar{\bf y}, \bar{\bf u}, \bar{\bf v}) \in  \prod_{i=1}^{k} S^{m_i}\times \prod_{j=1}^{\ell} S^{n_j} \times \prod_{i=1}^{k}\RR^{m_i+1} \times \prod_{j=1}^{\ell} \RR^{n_j+1}~{\big |}~{\bf u}_i \perp {\bf x}_i~ \& ~ {\bf v}_j \perp {\bf y}_j\} $$
 where $\bar{\bf x}=({\bf x}_1, \ldots, {\bf x_k})$, $\bar{\bf y} = ({\bf y}_1, \ldots, {\bf y}_{\ell}),$ $\bar{\bf u}=({\bf u}_1, \ldots, {\bf u_k})$, $\bar{\bf v} = ({\bf v}_1, \ldots, {\bf v}_{\ell}),$ $1 \leq i \leq k$ and $1 \leq j \leq \ell$.  We consider the equivalence  
  relation $\sim$ on the tangent space $T(  S(m_1, \ldots, m_k)  \times S(n_1, \ldots, n_{\ell}) )$ defined by $( \bar{\bf x}, \bar{\bf y}, \bar{\bf u}, \bar{\bf v})  \sim (-\bar{\bf x}, \sigma(\bar{\bf y}), -\bar{\bf u}, \sigma(\bar{\bf v}))$, where $\sigma(\bar{\bf v})=(\sigma_1,\ldots,\sigma_\ell)(\bar{\bf v})=(\sigma_1({\bf v}_1),\ldots,\sigma_\ell({\bf v}_\ell))$ and the actions $\sigma_j$'s are defined in \eqref{eq:multiple_reflection}. 

So the tangent space on $P(m_1, \ldots, m_k; (n_1, p_1), \ldots, (n_\ell, p_\ell))$ is given by the equivalence classes $$\displaystyle\{[ \bar{\bf x}, \bar{\bf y}, \bar{\bf u}, \bar{\bf v}]  \colon ( \bar{\bf x}, \bar{\bf y}, \bar{\bf u}, \bar{\bf v}) \in T( S(m_1, \ldots, m_k)  \times  S(n_1, \ldots, n_{\ell}))\}.$$ We have the following isomorphism of bundles 
$$ T( \displaystyle \prod_{i=1}^k S^{m_i} \times \displaystyle \prod_{i=1}^{\ell} S^{n_j}) \oplus (k+\ell)\epsilon \cong 
\sum_{i=1}^k (m_i+1) \epsilon \oplus \sum_{j=1}^\ell (p_j-1) \epsilon \oplus \sum_{j=1}^{\ell}(n_j-p_j+2)\epsilon $$ where $\epsilon$ represent the trivial line bundle on $ S(m_1, \ldots, m_k)  \times  S(n_1, \ldots, n_\ell) $. Now consider the natural $\ZZ_2$-actions on the both sides where $\ZZ_2$ acts on $ T( S(m_1, \ldots, m_k)  \times  S(n_1, \ldots, n_{\ell}))$ defined by $-1 (\bar{\bf x}, \bar{\bf y}, \bar{\bf u}, \bar{\bf v}) = (-\bar{\bf x}, \sigma(\bar{\bf y}), -\bar{\bf u}, \sigma(\bar{\bf v}))$, on $(k+\ell)\epsilon$ trivially, on each $(m_i+1)\epsilon$ antipodally, on each $(p_j-1)\epsilon$ trivially and on each $(n_j-p_j+2)\epsilon$ antipodally. This implies that the following bundle map is $\ZZ_2$-equivariant. 
\begin{center}
\begin{tikzcd}
T( \displaystyle \prod_{i=1}^k S^{m_i} \times \displaystyle \prod_{i=1}^{\ell} S^{n_j}) \oplus (k+\ell)\epsilon \arrow{d}{} \arrow{r}{\cong}
&  \displaystyle \sum_{i=1}^k (m_i+1) \epsilon \oplus \sum_{j=1}^\ell (p_j-1) \epsilon \oplus \sum_{j=1}^{\ell}(n_j-p_j+2)\epsilon \arrow{d}{} \\
S(m_1, \ldots, m_k)  \times  S(n_1, \ldots, n_\ell) \arrow{r}{=}
&  S(m_1, \ldots, m_k)  \times  S(n_1, \ldots, n_\ell)
\end{tikzcd}
\end{center}
This is induced from the $\ZZ_2$-action on  $ S(m_1, \ldots, m_k)  \times  S(n_1, \ldots, n_\ell) $. So  we get that the stable tangent bundle $\displaystyle T(P(m_1,\ldots,m_k; (n_1, p_1), \ldots, (n_\ell, p_\ell)))\oplus (k+\ell)\epsilon$ is isomorphic to the bundle 
$$\big{(}\sum_{i=1}^k (m_i+1) \oplus \sum_{j=1}^{\ell}(n_j-p_j+2)\big{)}\eta_{\ell+1} \oplus\sum_{j=1}^\ell (p_j-1) \epsilon. $$
Here the line bundle $\eta_{\ell+1}$ is defined on  $P(m_1, \ldots, m_k; (n_1, p_1), \ldots, (n_\ell, p_\ell))$ for $\ell \geq 0$, see Subsection \ref{subsec:proj_prod_sp}. We note that the stable tangent bundle of $P(m_1, \ldots, m_k)$ is discussed in \cite{Davis}. 

Next result gives an upper bound for the immersion dimension of some $P(m_1, \ldots, m_k; (n_1, p_1), \ldots, (n_\ell, p_\ell))$. 
We recall that the geometric dimension of a vector bundle $\eta \colon E \to M$ is the smallest positive integer $\mbox{gd}(\eta)$ such that $\eta$ is stably isomorphic to a vector bundle of rank $\mbox{gd}(\eta)$ on $M$.    
 The arguments of the proof of the following proposition is similar to the proof of \cite[Theorem 3.4]{Davis}, so we omit the details.   
\begin{prop}\label{prop:immersion}
Let $m_1 \leq m_i$ and $m_1 \leq n_j$ for $ 1 \leq i \leq k$ and $1 \leq j \leq \ell$. Then $\mbox{imm}(P(m_1,\ldots,m_k; (n_1, p_1), \ldots, (n_\ell, p_\ell)))$ is  equal to $$ {\bf d} +\max\{ \mbox{gd}((-({\bf d} + k + 2 \ell + {\bf p})) \eta_1), 1 \}$$ where ${\bf d} = \dim ((P(m_1,\ldots,m_k; (n_1, p_1), \ldots, (n_\ell, p_\ell)))$ and ${\bf p} = \sum_{j=1}^{\ell}p_j$. 
\end{prop}

\begin{remark}\label{whitney_class_1}
From the stable tangent  bundle isomorphism, one can compute the total Stiefel-Whitney class of  $P(m_1,\ldots,m_k; (n_1, p_1), \ldots, (n_\ell, p_\ell))$ where the Stiefel-Whitney class of $\eta_\ell$ is given by $w(\eta_\ell)=(1+w_1(\eta_\ell))$.
\end{remark}


\subsection{Stable tangent bundle on some manifolds in Example \ref{ex:toric}}\label{sec:toric_bundle}
In this subsection, we construct several line and plane bundles on the generalized projective spaces defined in  Example \ref{ex:toric}. Then we show that a stable tangent bundle of $P(m_1, \ldots, m_k; X)$ is a Whitney sum of these bundles where $X = \CC P^{n_1} \times \cdots \times \CC P^{n_\ell}$.  

Let $X$ be a toric manifold over a simple polytope $Q$ which has $\mu$ many facets. Consider the line bundle $ S(m_1, \ldots, m_k) \times X \times \mathbb{R}  \to  S(m_1, \ldots, m_k)  \times X $ and the $\ZZ_2$-action on the total space defined by $$({\bf x}_1, \ldots, {\bf x}_k, y, t) \mapsto (-{\bf x}_1, \ldots, -{\bf x}_k, \tau(2)(y), -t)$$ where $\tau(2)$ as in Example \ref{ex:conjugation_toric}. So the bundle map is $\ZZ_2$-equivariant, and it induces the following line bundle $$\eta \colon   \big{(} S(m_1, \ldots, m_k)  \times X \times \mathbb{R} \big{)}/\ZZ_2  \to  P(m_1, \ldots, m_k;  X). $$ The fixed point set of the involution $\tau(2)$ on $X$ is non-empty by definition.  Let $y_0$ be a fixed point of this involution on $X$. Then  $  S(m_1, \ldots, m_k)  \times y_0  \subseteq   S(m_1, \ldots, m_k)  \times X$. This gives an inclusion $P(m_1, \ldots, m_k) \subseteq P(m_1, \ldots, m_k; X)$ so that the following diagram commutes.  
\begin{equation}\label{eq:cano_line_bund}
\begin{tikzcd}
\big{(} S(m_1, \ldots, m_k)  \times \{y_0\} \times \mathbb{R} \big{)}/\ZZ_2 \arrow{d}{\eta_1} \arrow{r}{\overline{\iota}}
& \big{(} S(m_1, \ldots, m_k)  \times X \times \mathbb{R} \big{)}/\ZZ_2 \arrow{d}{\eta} \\
P(m_1, \ldots, m_k) \arrow{r}{\iota}
&  P(m_1, \ldots, m_k; X).
\end{tikzcd}
\end{equation}
 Then $\iota^*(\eta)=\eta_1$ where $\eta_1$ is defined in \eqref{eq:can_line_bundle} for $\ell=1$. By naturality, the total Stiefel-Whitney class of $\eta$ is given by $\omega(\eta)= 1+ \omega_1(\eta)$, where $c:=\omega_1(\eta_1) =\iota^*(\omega_1(\eta))$. 

Recall from Section \ref{sec:toric_small} that  $Z_Q(2)$ is the moment angle manifold corresponding to $X$ and the group $T^{m-n} = \ker({\rm exp}\Lambda)$ in \eqref{eq:torus_exact} is a  subtorus of $T(2)^{\mu}$ such that $Z_Q(2)/T^{m-n} \cong X$.   We denote the natural action  $$T^{m-n} \times Z_Q(2) \to Z_Q(2)$$ by $\nu$.  Let $\pi_i \colon T(2)^{\mu} \to S^1$ be the projection onto the $i$th factor. The torus $T(2)^{\mu}$ acts on $\CC$ via this projection by complex multiplication.  So $T^{m-n}$ acts on $\CC$ via the composition $T^{m-n} \hookrightarrow T(2)^{\mu} \xrightarrow{\pi_i} S^1$. We denote this one dimensional representation of $T^{m-n}$ by $\CC_i$ and the associated action by $\rho_i$ for $i=1, \ldots, \mu$.  

Now define an identification on $S(m_1, \ldots, m_k) \times Z_Q(2) \times \CC_i$ defined by 
\begin{equation}\label{eq:iden_bundle}
({\bf x}, {\bf y}, z) \sim (-{\bf x}, \tau(2){\bf y}, \overline{z})  \sim ({\bf x}, \nu({\bf y}), \rho_i(z)) \sim (- {\bf x}, \tau(2){\nu({\bf y})}, \overline{\rho_i(z)}),
\end{equation} 
here $ \overline{z}$ represents complex conjugation on $z$. Then the identification space gives a real 2-plane bundle  $$\zeta_i \colon (S(m_1, \ldots, m_k) \times Z_Q(2) \times \CC_i)/T^{m-n} \to P(m_1, \ldots, m_k; X)$$ on $P(m_1, \ldots, m_k; X)$, denoted by $\zeta_i$ for $i=1, \ldots, \mu$. 

Next, we recall some canonical 2-plane bundles on $X$ following \cite{BP}. The trivial complex line bundles $Z_Q(2) \times \CC_i \to Z_Q(2)$ is equivariant with respect to the action of $T^{m-n}$. This induces a real 2-plane bundle $$L_i \colon Z_Q(2) \times_{T^{m-n}} \CC_i \to X \cong Z_Q(2)/T^{m-n}.$$
The second  Stiefel-Whitney class of the bundle $L_i$ is given by $\omega_2(L_i) = u_i$-mod 2 where $u_i \in H^2(X; \ZZ)$ is the Poincar{\'e} dual to the characteristic submanifold $\mathfrak{q}^{-1}(F_i)$, see  \cite[Section 6]{DJ}, where $F_i$ is the $i$th facets of $Q$. Let $a=(e_{i_1}, \ldots, e_{1_k}) \in S(m_1, \ldots, m_k)~~$ where $e_{1_i}$ is the first vector in the standard basis of $\RR^{m_i+1}$. Then $a, -a \in S(m_1, \ldots, m_k)$. So we get the inclusion $$\iota_0 \colon X = P(\{a, -a\}, X) \subset P(m_1, \ldots, m_k; X).$$ Then the pull-back of $\zeta_i$ under $\iota_0$ is $L_i$  for $i=1, \ldots, \mu$. By naturality $\iota_0^*\omega_2(\zeta_i) = u_i$-mod 2 for  $i=1, \ldots, \mu$.

Let $X= \CC P^{n_1} \times \cdots \times \CC P^{n_\ell}$ in the remaining.  Then the moment angle manifold $Z_Q(2) $ corresponding to this $X$ is $S^{2n_1+1} \times \cdots \times S^{2n_{\ell}+1}$ and the corresponding $T^{m-n}$ can be identified with $(S^1)^{\ell}$. Also the action $\nu$ for this case is the coordinate wise action of    $(S^1)^{\ell}$ on   $S^{2n_1+1} \times \cdots \times S^{2n_{\ell}+1}$. For each $j \in \{1, \ldots, \ell\}$, the identification in \eqref{eq:iden_bundle} reduces to 
$$({\bf x}, {\bf y}, z) \sim (-{\bf x}, \overline{\bf y}, \overline{z})  \sim ({\bf x}, \nu({\bf y}), h_j z) \sim (- {\bf x}, \overline{\nu({\bf y})}, \overline{h_j z})$$ where $h_j$ belongs to the $j$th coordinate circle of $(S^1)^{\ell}$. This induces a real 2-plane bundle $\zeta'_j$ over $P_T(m_1, \ldots, m_k; n_1, \ldots, n_\ell)$. The fixed point set of the involution $\tau(2)$ on $X$ is non-empty. Therefore, the map $$S(m_1, \ldots, m_k) \times \prod_{j=1}^\ell S^{2n_j+1}  \times \RR^2 \to S(m_1, \ldots, m_k) \times \prod_{j=1}^\ell S^{2n_j+1} \times \CC$$ defined by $({\bf x}, {\bf y}, r_1, r_2) \mapsto ({\bf x}, {\bf y}, r_1 + \sqrt{-1}r_2)$ induces the following bundle map.  
\begin{equation}\label{eq:}
\begin{tikzcd}
E(\epsilon \oplus \eta_1) \arrow{d}{\eta_1'} \arrow{r}{\overline{\iota_j}}
& E(\zeta'_{j}) \arrow{d}{\zeta'_{j}} \\
P(m_1, \ldots, m_k) \arrow{r}{\iota}
&  P_T(m_1, \ldots, m_k; n_1, \ldots, n_\ell).
\end{tikzcd}
\end{equation}
where $E(\ast)$ represent the total space of the corresponding bundle and and $j=1, \ldots, \ell$.  So $\iota^*(\zeta'_j) =\epsilon \oplus \eta_1 $. On the other hand, $\iota_0^*(\zeta'_j)$ is a real 2-plane bundle over $\CC P^{n_1} \times \cdots \times \CC P^{n_\ell}$ for $j=1, \ldots, \ell$. Note that for this case, $$H^*(X; \ZZ_2) \cong \frac{\ZZ_2[d_1]}{d_1^{n_1+1}} \otimes \cdots \otimes  \frac{\ZZ_2[d_\ell]}{d_{\ell}^{n_\ell + 1}}$$ where $d_j$ is the canonical generator of $H^*(\CC P^{n_j}; \ZZ_2)$. If $m_1, \ldots, m_k$ are greater than 1, then by Theorem \ref{thm:toric_mod2} and the cohomology of $X$, we get $\omega_1(\zeta'_j)=c=\omega_1(\eta_1)$ and $\omega_2(\zeta'_j)=d_j$ for $j=1, \ldots, \ell$. 
 
Using the above discussion and the proof of  \cite[Theorem 1.5]{Ucci}, one gets the following. So we omit the details. 

\begin{thm}
The bundle $T(P_T(m_1, \ldots, m_k; n_1, \ldots, n_k)) \oplus \ell \eta \oplus (k+\ell) \epsilon$ is isomorphic to $ \sum_1^k(m_i+1) \eta \oplus 
(n_1+1) \zeta'_1 \oplus \cdots \oplus (n_\ell+1)\zeta'_\ell.$
\end{thm}

\begin{cor}
If $m_1, \ldots, m_k$ are greater than 1, then the total Stiefel-Whitney class of  $P_T(m_1, \ldots, m_k; n_1, \ldots, n_\ell)$ is given by $$W (P_T(m_1, \ldots, m_k; n_1, \ldots, n_k))=(1+c)^{(\sum_1^k m_i +k-\ell)}\displaystyle\prod_{j=1}^{\ell}(1+c+d_j)^{n_j+1}.$$
\end{cor}


\subsection{Stable tangent bundle on manifolds in Example \ref{ex:small}}\label{sec:small_bundle}
In this subsection, we construct several line bundles on the generalized projective spaces defined in  Example \ref{ex:small}. Then we show that  a stable tangent bundle of $P(m_1, \ldots, m_k; Y)$ is a Whitney sum of these bundles when $Y = \RR P^{n_1} \times \cdots \times \RR P^{n_\ell}$. Most of the discussion and calculation are similar to Subsection \ref{sec:toric_bundle} except few important observations related to small covers. We keep similar notation but they are contextual.   

Let $Y$ be a small cover over a simple polytope $Q$ which has $\mu$ many facets. By similar arguments as in the beginning of Subsection \ref{sec:toric_bundle} one can show the following. 
\begin{equation}\label{eq:cano_line_bund_sm}
\begin{tikzcd}
\big{(} S(m_1, \ldots, m_k)  \times \{y_1\} \times \mathbb{R} \big{)}/\ZZ_2 \arrow{d}{\eta_1} \arrow{r}{\overline{\iota}}
& \big{(} S(m_1, \ldots, m_k)  \times Y \times \mathbb{R} \big{)}/\ZZ_2 \arrow{d}{\eta} \\
P(m_1, \ldots, m_k) \arrow{r}{\iota}
&  P(m_1, \ldots, m_k; Y).
\end{tikzcd}
\end{equation}
Here $\eta_1$ is the canonical line bundle defined in \eqref{eq:can_line_bundle} for $\ell=1$. Then $\iota^*(\eta)=\eta_1$. By naturality, the total Stiefel-Whitney class of $\eta$ is given by $\omega(\eta)= 1+\omega_1(\eta)$, where $c=\omega_1(\eta_1) =\iota^*(\omega_1(\eta))$. 

Recall from Section \ref{sec:toric_small} that  $Z_Q(1)$ is the real moment angle manifold corresponding to $Y$ and the group $\ker(\Lambda)$ in \eqref{eq:smcv_exact} is a real subtorus of $T(1)^{\mu}$ such that $Z_Q(1)/\ker(\Lambda) \cong Y$. We denote the natural action  $$\ker(\Lambda) \times Z_Q(1) \to Z_Q(1)$$ by $\nu$.  Let $\pi_i \colon T(1)^{\mu} \to \ZZ_2$ be the projection onto the $i$th factor. The torus $T(1)^{\mu}$ acts on $\RR$ via this projection.  So $\ker(\Lambda)$ acts on $\RR$ via the composition $\ker(\Lambda) \hookrightarrow T(1)^{\mu} \xrightarrow{\pi_i} \ZZ_2$. We denote this one dimensional representation of $\ker(\Lambda)$ by $\RR_i$ and the associated action by $\rho_i$ for $i=1, \ldots, \mu$.  

Now define an identification on $S(m_1, \ldots, m_k) \times Z_Q(1) \times \RR_i$ defined by 
\begin{equation}\label{eq:iden_bundle_sm}
({\bf x}, {\bf y}, z) \sim (-{\bf x}, \tau(1){\bf y}, -z)  \sim ({\bf x}, \nu({\bf y}), \rho_i(z)) \sim (- {\bf x}, \tau(1){\nu({\bf y})}, -\rho_i(z)).
\end{equation} 
Then the identification space gives a line bundle  $$\zeta_i \colon (S(m_1, \ldots, m_k) \times Z_Q(1) \times \RR_i)/\ker(\Lambda) \to P(m_1, \ldots, m_k;  Y)$$ on $P(m_1, \ldots, m_k; Y)$, denoted by $\zeta_i$ for $i=1, \ldots, \mu$. 

Next, we recall some canonical line bundles on $Y$ following \cite{DJ}. The trivial line bundle $Z_P(1) \times \RR_i \to Z_P(1)$ is equivariant with respect to the action of $\ker(\Lambda)$. This induces a line bundle $$L_i \colon Z_Q(1) \times_{\ker(\Lambda)} \RR_i \to X \cong Z_Q(1)/\ker(\Lambda).$$
The  Stiefel-Whitney characteristic class of the line bundle $L_i$ is given by $\omega(L_i) = 1+ u_i$ where $u_i \in H^1(Y; \ZZ_2)$ is the Poincar{\'e} dual to the characteristic submanifold $\mathfrak{q}^{-1}(F_i)$, see  \cite[Section 6]{DJ}, where $F_i$ is the $i$th facets of $Q$. Let $a=(e_{i_1}, \ldots, e_{1_k}) \in S(m_1, \ldots, m_k)$ where $e_{1_i}$ is the first vector in the standard basis of $\RR^{m_i+1}$. Then $a, -a \in S(m_1, \ldots, m_k)$. So we get the inclusion $$\iota_0 \colon Y = P(\{a, -a\}, Y) \subset P(m_1, \ldots, m_k; Y).$$ Then the pull-back of $\zeta_i$ under $\iota_0$ is $L_i$  for $i=1, \ldots, \mu$. By naturality $\iota_0^*\omega_1(\zeta_i) = u_i$ for  $i=1, \ldots, \mu$.

Let $Y= \RR P^{n_1} \times \cdots \times \RR P^{n_\ell}$ in the remaining.  Then the real moment angle manifold $Z_Q(1) $ corresponding to this $Y$ is $S^{n_1+1} \times \cdots \times S^{n_{\ell}+1}$ and the corresponding $\ker(\Lambda)$ can be identified with $(\ZZ_2)^{\ell}$. Then the action $\nu$ for this case is the coordinate wise action of  $(\ZZ_2)^{\ell}$ on   $S^{n_1+1} \times \cdots \times S^{n_{\ell}+1}$ where $\ZZ_2$ acts on each $S^{n_j+1}$ antipodally. For each $j \in \{1, \ldots, \ell\}$, the identification in \eqref{eq:iden_bundle_sm} reduces to 
$$({\bf x}, {\bf y}, z) \sim (-{\bf x}, -{\bf y}, -z)  \sim ({\bf x}, \nu({\bf y}), h_j z) \sim (- {\bf x}, -{\nu({\bf y})}, -{h_j z})$$ where $h_j$ belongs to the $j$th coordinate of $(\ZZ_2)^{\ell}$. This induces a line bundle $\zeta'_j$ over $P_S(m_1, \ldots, m_k; n_1, \ldots, n_\ell)$. Then the pull-back of f $\zeta'_j$ under $\iota_0$ is $L_{i_j}$  for some $i_j \in \{1, \ldots, \mu\}$. 
 Note that for this case, $$H^*(Y; \ZZ_2) \cong \frac{\ZZ_2[d_1]}{d_1^{n_1+1}} \otimes \cdots \otimes  \frac{\ZZ_2[d_\ell]}{d_{\ell}^{n_\ell + 1}}$$ where $d_j$ is the canonical generator of $H^*(\RR P^{n_j}; \ZZ_2)$. If $m_1, \ldots, m_k$ are greater than 1, then by Proposition \ref{prop:smcv_mod2} and the cohomology of $Y$, we get $\omega_1(\zeta'_j)=d_j$ for $j=1, \ldots, \ell$. 
 
Using the above discussion and the proof of  \cite[Theorem 1.5]{Ucci}, one can also gets the following. So we omit the details. 

\begin{thm}
The bundle $T(P_S(m_1, \ldots, m_k; n_1, \ldots, n_k)) \oplus (k+\ell) \epsilon$ is isomorphic to $ \sum_1^k(m_i+1) \eta \oplus 
(n_1+1) \zeta'_1 \oplus \cdots \oplus (n_\ell+1)\zeta'_\ell.$
\end{thm}

\begin{cor}
If $m_1, \ldots, m_k$ are greater than 1, then the total Stiefel-Whitney class of  $P_S(m_1, \ldots, m_k; n_1, \ldots, n_\ell)$ is given by $$W (P_S(m_1, \ldots, m_k; n_1, \ldots, n_k))=(1+c)^{\sum_1^k (m_i+1)}\displaystyle\prod_{j=1}^{\ell}(1+d_j)^{n_j+1}.$$
\end{cor}


\section{Bounds for the span of the spaces in Section \ref{sec:example_proj_prod_spaces}}\label{sec:span}
Let $X$ be a finite dimensional smooth manifold and the map $$\pi \colon TX \to X$$ the tangent bundle on $X$. The vector field problem studies the tangent bundle $\pi$  by seeking continuous nowhere zero sections $s : X \to TX$ of $\pi$. The maps $\{s\}$ are called non-zero vector fields on $X$.   The maximum number of pointwise linearly independent vector fields on the manifold $X$ is called the  span of $M$ and it is denoted by $\mbox{sp}(M)$. It is an invariant of topological type. The celebrated work   \cite{Ada} of Adams gives the complete solution of the vector field problem for any sphere $S^n$ and any projective space $\mathbb R P^n$. However, in general, these problems are open for most other smooth manifolds, e.g. even on the product of two real projective spaces \cite{DD}. We refer the reader to \cite{KZ} and \cite{Sankaran} for some motivational background on the vector field problems. We note that if $S^m$ admits an $r$-field, it also admits a linear $r$-field which is equivariant with respect to antipodal action. An even stronger result is shown in \cite{MZ} by Milgram and the second author, that every $r$-field on any sphere is homotopic to an $r$-field that is equivariant with respect to the antipodal action. Novotn\'y in \cite{Nov} showed that if $S^m$ admits $k$ many linearly independent vector fields which are equivariant with respect to the antipodal action, then the Dold manifold $D(m,1)$ admits at least $k+1$ many linearly independent vector fields.  

For the generalized projective product spaces or generalized Dold manifolds, we have the following observation.  If $M, N$ are $\ZZ_2$-spaces and $\ZZ_2$ acts on $M$ freely, then the Euler characteristic of $P(M, N)$ is given by
\begin{equation}
\mathcal{X}(P(M, N) =\frac{1}{2} \mathcal{X}(M) \mathcal{X}(N).
\end{equation}
So $\mbox{sp}(P(M, N)) \geq 1$ if and only if one of $\mathcal{X}(M), \mathcal{X}(N)$ is zero, by \cite[Theorem 1.7]{Sankaran}.  Since $\mbox{sp}(E) \geq \mbox{sp}(B)$ for a smooth fibre bundle  $F \hookrightarrow E \to B$,
 then we have $\mbox{sp}(P(M, N)) \geq \mbox{sp}(M/ \mathbb{Z}_2)$. 

We recall that if $X$ is a smooth manifold then the stable span of $X$ is the maximum integer $r$ such that $TX +k \epsilon \cong (k+r) \epsilon + \eta$ for some $k \geq 1$ and a bundle $\eta$ on $X$. The manifold $X$ is called stably parallelizable if $TX +k \epsilon$ is trivial for some $k \geq 1$. We denote the stable span of $X$ by ${\rm stasp}(X)$. 

In this section, we compute lower and upper bounds for the span and stable span of several generalized projective product spaces defined in Section \ref{sec:example_proj_prod_spaces} and improve the above lower bound for certain generalized projective product spaces.   In particular, we extend the result of \cite{Nov} to a broader class of manifolds. 
   
\subsection{Vector fields on manifolds in Example \ref{ex:proj_prod_sp}}
In this subsection, we study the vector fields problems on the generalized projective product spaces defined in   Example \ref{ex:proj_prod_sp}. We adhere the notation of  Example \ref{ex:proj_prod_sp}. Note that the corresponding orbit map $S(m_1, \ldots, m_k) \times S(n_1, \ldots, n_\ell) \to P_{\overline{m},  (\overline{n}, \overline{p})} $ is a double covering. 
So the  Euler characteristic of these generalized  projective product spaces can be given by
\begin{equation}
\mathcal{X}(P_{\overline{m}, (\overline{n}, \overline{p})}) = \left\{ \begin{array}{ll} 2^{k+\ell-1} & \mbox{if all}~ m_i, n_j ~ \mbox{are even} \\
 0 & \mbox{if one of}~ m_i ~\mbox{or}~ n_j~ \mbox{is odd}.
\end{array} \right.
\end{equation}
Then, by \cite[Theorem 20.1]{Kos} we have the following.
\begin{prop}
If $\sum_1^k m_i + \sum_1^\ell n_j $ is even and at least one of $m_i$ or $n_j$ is odd, then $\emph{sp}(P_{\overline{m}, (\overline{n}, \overline{p})})=\emph{stasp}(P_{\overline{m}, (\overline{n}, \overline{p})}).$
\end{prop}

Note that $\mbox{sp}(P_{\overline{m}, (\overline{n}, \overline{p})}) = 0 $ if and only if  all $m_i, n_j$ are even. Consider the antipodal action of $\ZZ_2$ on  each $S^{m_i}$ and $S^{n_j}$. Then the coordinate wise action of $\ZZ_2^k \times \ZZ_2^{\ell}$ on the product $S(m_1, \ldots, m_k) \times S(n_1, \ldots, n_\ell)$ is free with the orbit space $\prod_{i=1}^k \RR P^{m_i} \times \prod_{j=1}^\ell \RR P^{n_j}$. So the following maps are covering.
\begin{equation}\label{double1}
S(m_1, \ldots, m_k) \times S(n_1, \ldots, n_\ell)  \to P_{\overline{m}, (\overline{n}, \overline{p})} \to \prod_{i=1}^k \RR P^{m_i} \times \prod_{j=1}^\ell \RR P^{n_j}.
\end{equation}
Therefore, if $|\overline{m}| = m_1 + \cdots +m_k,$ $|\overline{n}| = n_1 + \cdots +n_\ell$, and  at least one of $m_i$ or $n_j$ is odd, then $$\sum_1^k \mbox{sp}(S^{m_i}) + \sum_1^\ell \mbox{sp}(S^{n_j}) \leq \mbox{sp} (\prod_{i=1}^k \RR P^{m_i} \times \prod_{j=1}^\ell \RR P^{n_j}) \leq \mbox{sp}(P_{\overline{m}, (\overline{n}, \overline{p})}) \leq |\overline{m}| + |\overline{n}|.$$

The parallelizability problem for projective product spaces has been completely solved in \cite[Theorem 3.12]{Davis}. One can ask this problem for the generalized projective product space $P(m_1, \ldots, m_k; (n_1, p_1), \ldots, (n_\ell, p_\ell)) $ when $1 \leq p_j \leq n_j$ for $j=1, \ldots, \ell$. From the definition of this space, we have the fibre bundle: $$S(n_1, \ldots, n_\ell) \hookrightarrow P_{\overline{m}, (\overline{n}, \overline{p})} \to P_{\overline{m}}$$ where $P_{\overline{m}}$ is a projective product space.  Thus $\mbox{sp}(P_{\overline{m}, (\overline{n}, \overline{p})}) \geq \mbox{sp}(P_{\overline{m}})$. In the remaining of this subsection we improve this lower bound.

\begin{thm}\label{thm:gnrproj_prod_sps}
If $m$ is odd and $p \geq 1$, then $\emph{sp}(P(m; (n, p)) \geq \emph{sp}(S^m) + p-1$.  
\end{thm}
\begin{proof}
Since $S^m$ is an odd sphere, then by a theorem of Hopf, we have $r = \mbox{sp}(S^m) \geq 1$. So by \cite{MZ}, $S^m$  admits  $r$ many linearly independent vector fields which are equvariant with respect to the antipodal action. Let  $v_1, \ldots, v_r$ be $r$ many linearly independent $\ZZ_2$-equivariant vector fields on $S^m$.
Now we define the vector fields $w_1 \ldots, w_{r+p-1}$ on $S^m\times S^n$ as follows.
\begin{align*}
& w_i(\overline{\bf x},(y_1,\ldots,y_{n+1})) =\\
&                                  \begin{cases}
                                         (v_i(\overline{\bf x}), (0,\ldots,0)), & \text{if} ~ 1 \leq i \leq r-1\\
                                         (y_j v_r(\overline{\bf x}),(y_1 y_j, \ldots, y_{j-1} y_j, y_j^2-1, y_{j+1}y_j, \ldots, y_{n+1}y_j), & \text{if}~ r \leq i \leq r+p-1 
\end{cases}
\end{align*}
where $j=i-r+1$.

We show that $w_1, \ldots, w_{r+p-1}$ are linearly independent at each point on  $S^m\times S^n$.
Suppose there are scalars $b_1, \ldots, b_{r-1}, a_1, \ldots, a_p$ such that $$\sum_{i=1}^{r-1} b_i w_i (\overline{\bf x}, \overline{\bf y}) +\sum_{j=1}^{p}a_j w_{r+j-1}(\overline{\bf x}, \overline{\bf y}) = 0$$  at some point $(\overline{\bf x}, \overline{\bf y}) \in S^m \times S^n$. Since $v_1, \ldots, v_r$ are linearly independent, we get the following. 
$$\sum_{j=1}^{p} a_j y_j=0 \quad \mbox{and} \quad b_i = 0$$
for $i=1, \ldots, r-1$. Also 
\begin{align*}
\begin{split}
0=a_1 y_1^2 - a_1 + \sum_{j=2}^{p} a_j y_1 y_j = y_1 (\sum_{j=1}^{p} a_j y_j) - a_1, &\\
 \vdots \quad \quad \quad \quad \quad \quad \quad \quad  \\
0= a_p y_p^2 - a_p + \sum_{j=1}^{p-1} a_j y_p y_j = y_{p} (\sum_{j=1}^{p} a_j y_j) - a_p. 
\end{split}
\end{align*} 
From the above equations, we also get $a_j = 0$ for $j= 1, \ldots, p$. Therefore $w_1,\ldots,w_{k+p-1}$ are linearly independent vector fields on $S^m \times S^n$.

To show that  $w_1, \ldots,w_{k+p-1}$ are equivariant vector fields under the $\ZZ_2$-action on $S^m \times S^n$ as defined in  Example \ref{ex:proj_prod_sp}, one can argue similarly as in the proof of \cite[Theorem 4.2]{Nov}. 
\end{proof}

\begin{cor}
If one of $m_1, \ldots, m_k $ is odd and $1 \leq p_j \leq n_j$ for $j= 1, \ldots, \ell$, then $\emph{sp} (P_{\overline{m}, (\overline{n}, \overline{p})}) \geq \emph{sp}(P_{\overline{m}}) + \sum_{j=2}^{\ell}(p_j-1)$. 
\end{cor}

\begin{proof}
Let $r = \mbox{sp}(P_{\overline{m}})$. Since $S(m_1, \ldots, m_k) \to P_{\overline{m}}$ is a double covering, obtained from the antipodal action on the spheres, hence there are $\ZZ_2$-equivariant $r$ many linearly independent vector fields on $S(m_1, \ldots, m_k)$. Using the fact that $P_{\overline{m}, (\overline{n}, \overline{p})}$ is an iterated sphere bundle over $P_{\overline{m}}$ and applying  Theorem \ref{thm:gnrproj_prod_sps} repeatedly one can get the corollary. 
\end{proof}


\subsection{Vector fields on the manifolds in Example \ref{ex:toric}}
Let $M$ be a manifold equipped with a free $\ZZ_2$-action and $X$ a toric manifold over a simple polytope $P$ equipped with an involution as considered in Example \ref{ex:conjugation_toric}. In this subsection we study vector fields problems on $P(M, X)$ when $M=S(m_1, \ldots, m_k)$ and $M$ has some $\ZZ_2$ equivariant linearly independent vector fields. Let  $V(Q)$ be the set of vertices of a simple polytope $Q$. Then $\mathcal{X}(X) = |V(Q)|$, see \cite{DJ}. Hence by \eqref{double2} the Euler characteristic of $P(M, X)$ is given by

\begin{equation}
\mathcal{X}(P(M, X) =\frac{1}{2} \mathcal{X}(M) \mathcal{X}(X) = \frac{1}{2} \mathcal{X}(M) |V(Q)|
\end{equation}
In particular, since $\mathcal{X}(S^m)=1+(-1)^m$, we get

\begin{equation}\label{euler_char2}
\mathcal{X}(P(S(m_1, \ldots, m_k), X)) = \left\{ \begin{array}{ll} 2^{k-1} |V(Q)| & \mbox{if}~ m_1, \ldots, m_k ~ \mbox{are even} \\
 0 & \mbox{if one}~ m_i ~ \mbox{is odd}.
\end{array} \right.
\end{equation}
Since \eqref{fiber2} is a smooth fibre bundle, by \cite[ 3.1.6 (1) ]{KZ}, we have the following,
\begin{equation}\label{lower_bound}
\mbox{stasp}(P(S(m_1, \ldots, m_k), X)) \geq \mbox{sp}(P(S(m_1, \ldots, m_k),  X)) \geq \mbox{sp}(P_{\overline{m}}).
\end{equation}
where $\overline{m}=(m_1, \ldots, m_k)$.

Next proposition and corollary discuss the stable parallelizability of the manifolds $P(M, X)$.

\begin{prop}
If the manifold $P_T(m_1, \ldots, m_k; n_1, \ldots, n_\ell)$ is stably parallelizable,  then $n_j \in \{0, 1\}$ for $j=1, \ldots, \ell$. 
\end{prop}
\begin{proof}
Since $P(m_1, \ldots, m_k; n_1, \ldots, n_\ell)$ is stably parallelizable, then its double cover $\prod_{i=1}^{k} S^{m_i} \times \prod_{j=1}^{\ell} \CC P^{n_j}$ is so. Each $S^{m_i}$ is stably parallelizable. The product $\prod_{j=1}^{\ell} \CC P^{n_j}$ is stably parallelizable if and only if each $n_j \in \{0, 1\}$, as the first Pontrjagin class $p_1(\prod_{j=1}^{\ell} \CC P^{n_j}) = \sum_1^\ell (n_j+1) a_j^2$, where $a_j$ is the canonical generator of $H^*(\CC P^{n_j}; \ZZ)$. 
\end{proof}

Moreover, by similar arguments, we can get the following.
\begin{prop}
If the first Pontryagin class of $X$ is non-zero in $H^*(X; \ZZ)$, then $P(M, X)$ is not parallelizable. 
\end{prop}

We recall the cohomology ring structure of $X$ from \eqref{eq:cohom_toric_mfds}. By \cite[Corollary 6.8]{DJ}, the first Pontryagin class of $X$ is given by $p_1(X)=\sum_{i=1}^\mu u_i^2$ where $u_i$'s are the generators in \eqref{eq:cohom_toric_mfds}. If $\sum_{i=1}^\mu u_i^2 \neq 0$, then $X$ is not stably parallelizable. Also  note that  $\sum_{i=1}^\mu u_i^2= 0$ if $X$ is  the product of some $\CC P^1$. In this case $P(S(m_1 \ldots, m_k), \CC P^1 \times \cdots \times \CC P^1)$ is $P(m_1, \ldots, m_k; (2,2), \ldots, (2,2))$ of Example \ref{ex:proj_prod_sp} where the number of $(2,2)$ is same as the number of $\CC P^1$. Now we compute the first Pontryagin class of the 4-dimensional toric manifold $X^4$ over a square in the following.

\begin{example}\label{ex:comp_pont_class}
Let $X^4$ be a toric manifold over a square. So the characteristic function is given either by Figure \ref{Fig2sq} (a) or by Figure \ref{Fig2sq} (b) up to sign, see \cite[Example 1.19]{DJ}.\\
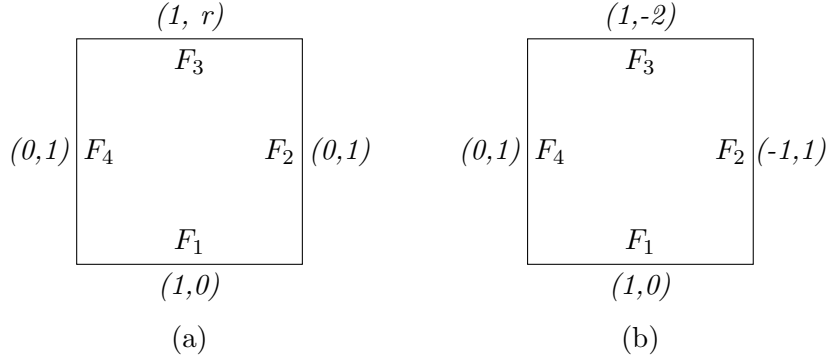
\begin{figure}
\begin{tikzpicture}

\draw(0,0) rectangle (3,3);
\node at (1.5,-1){\mbox{(a)}};
\node at (1.5,-0.3){\textit{(1,0)}};
\node at (-0.5,1.5){\textit{(0,1)}};
\node at (3.5,1.5){\textit{(0,1)}};
\node at (1.5,3.3){\textit{(1, r)}};
\node at (1.5,0.3){\textit{$F_1$}};
\node at (0.3,1.5){\textit{$F_4$}};
\node at (2.7,1.5){\textit{$F_2$}};
\node at (1.5,2.7){\textit{$F_3$}};
\end{tikzpicture}
\qquad
\begin{tikzpicture}
\draw(0,0) rectangle (3,3);
\node at (1.5,-1){\mbox{(b)}};
\node at (1.5,-0.3){\textit{(1,0)}};
\node at (-0.5,1.5){\textit{(0,1)}};
\node at (3.5,1.5){\textit{(-1,1)}};
\node at (1.5,3.3){\textit{(1,-2)}};
\node at (1.5,0.3){\textit{$F_1$}};
\node at (0.3,1.5){\textit{$F_4$}};
\node at (2.7,1.5){\textit{$F_2$}};
\node at (1.5,2.7){\textit{$F_3$}};
\end{tikzpicture}
\caption{Characteristic functions on a square.} \label{Fig2sq}
\end{figure}
For Figure \ref{Fig2sq} (a) we have the following from the definition of $J$ in \eqref{eq:cohom_toric_mfds}. 
$$ x_1+x_3=0 , \quad \mbox{and} \quad x_2+rx_3+x_4=0.$$
So,
\begin{align*}
x_1^2+x_2^2+x_3^2+x_4^2 &= 2x_1^2+2x_2^2+r^2x_3^2+2rx_2x_3\\
&=(2+r^2)x_1^2+2x_2^2-2rx_1x_2,\\
&=2x_2^2-2rx_1x_2, ~~~\text{as}~x_1^2=0.
\end{align*}

If $r=0$, then $x_2^2=0$, otherwise $x_2^2=-rx_1x_2\neq0$. So, $x_1^2+x_2^2+x_3^2+x_4^2\neq0$ if and only if $r \neq 0$. Hence $P(M, X^4)$ is not stably parallelizable if $r \neq 0$. For $r = 0$, $X^4=\CC P^1 \times \CC P^1$, and therefore $P(S^m, \CC P^1 \times \CC P^1)$ is stably parallelizable if $m=1, 3, 7$.  

For Figure \ref{Fig2sq} (b), $X^4$ is $\CC P^2 \# \CC P^2$. Then we have
$$ x_1-x_2+x_3=0 \Rightarrow x_3=x_2-x_1,$$
and
$$ x_2-2x_3+x_4=0 \Rightarrow x_4=-x_2+2(x_2-x_1) =x_2-2x_1.$$
Therefore,
\begin{align*}
x_1^2+x_2^2+x_3^2+x_4^2 &= x_1^2+x_2^2+(x_1-x_2)^2+(2x_1-x_2)^2\\ &=2x_1^2+2x_2^2-2x_1x_2+4x_1^2-4x_1x_2+x_2^2\\ &=6x_1^2+3x_2^2-6x_1x_2
\end{align*}

This is non-zero in $H^4(\CC P^2\#\CC P^2)$. Therefore $P(M, \CC P^2\#\CC P^2)$ is not stably parallelizable for any $M$ equipped with a free $\ZZ_2$-action.
\end{example}

In the remaining we improve the lower bound of \eqref{lower_bound} under some hypothesis. 
\begin{thm}\label{thm:one_plus_span}
Let $v_1, \ldots, v_k\colon M \mapsto TM $ be $\ZZ_2$-equivariant  linearly independent vector fields where $\ZZ_2$ acts freely on $M$. Then 
${\emph {sp}} (P(M, \CC P^1)) \geq k+1$. 
\end{thm}
\begin{proof}
This result can be obtained by similar arguments as in the proof of \cite[Theorem 4.2]{Nov}. We only write the vector fields. Identify $\CC P^1$ with $S^2$ as in \cite[Proposition 4.1]{Nov}.  
We define $(k+1)$ many vector fields on $M \times \CC P^1$ as follows:
\begin{align*}
w_i(x, (y_1, y_2, y_3)) &= ((x, (y_1, y_2, y_3)); (v_i(x), (0, 0, 0))), ~\text{for} ~1\leq i\leq k-1,\\
w_k (x, (y_1, y_2, y_3))& =((x, (y_1, y_2, y_3)), (y_1v_k(x), (y_1^2-1, y_1y_2, y_1y_3))), \\
w_{k+1} (x, (y_1, y_2, y_3))& =((x, (y_1, y_2, y_3)), (y_2v_k(x), (y_1y_2, y_2^2-1, y_2y_3))).
\end{align*}
Using Novotn\'y's argument, one can show that these vectors are $\ZZ_2$-invariant linearly independent vector fields on $M \times \CC P^1$.
\end{proof}

\begin{cor}
Let $v_1, \ldots, v_k\colon M \mapsto TM $ be $\ZZ_2$-equivariant  linearly independent vector fields where $\ZZ_2$ acts freely on $M$. Then ${\emph {sp}} (P(M, \displaystyle\prod_{j=1}^{\ell} \CC P^{1})) \geq r+\ell $. 
\end{cor}
\begin{proof}
Since $\ZZ_2$ acts freely on $M$, hence it acts freely on $M\times \displaystyle\prod_{j=1}^{\ell-1} \CC P^{1}$ and hence the corollary follows from Theorem \ref{thm:one_plus_span}.
\end{proof}

\begin{cor}
If one $m_i$ is odd in $\{m_1, \ldots, m_k\}$, then 
\begin{equation}\label{eq:sp_mn}
 \emph{sp}(P(S(m_1,\ldots,m_k), \prod_{j=1}^{\ell} \CC P^{1})) \geq {\emph{sp}} (P_{\overline{m}}) + \ell. 
\end{equation}
\end{cor}

We note that that span and stable of $P(m; 1) (=P(S^m, \CC P^1))$ is completely determined by Novotn\'y \cite{Nov} and Korba\v{s} \cite{Korbas}.
At this point we do not know if the equality in \eqref{eq:sp_mn} holds in general.

{\bf Acknowledgement:} The authors thank the referees for several helpful comments and suggestions.  The first author thanks University of Calgary where this work was initiated during his post doctoral position. The authors want to thank Sudeep Poddar for some help in Latex.


\noindent
\author{}\\
{\small }\\


\noindent
\author{Soumen Sarkar}\\
\small {Department of Mathematics}\\
\small{Indian Institute of Technology, Madras} \\
\small {Chennai, Tamil Nadu-600036, India } \\
\small {e-mail: soumensarkar20@gmail.com}\\

\noindent
\author{Peter Zvengrowski} \\
\small {Department of Mathematics and Statistics} \\
\small{University of Calgary} \\
\small {Calgary, Alberta, Canada T2N 1N4}\\
\small { e-mail: zvengrow@gmail.com}

\end{document}